%% file: MultNoise.tex
\documentclass[10pt]{article}
\usepackage{cite}
\usepackage{epsfig}
\usepackage[english]{babel}
\usepackage{graphicx,enumerate,algorithm,algorithmic,multirow,multicol}
\input epsf

\usepackage{amsmath}
\usepackage{amsfonts}
\usepackage{amsbsy}
\usepackage{amssymb}
\usepackage{psfrag}
\usepackage{color}
\newcommand {\defeq}{\stackrel{\rm def}{=}}

\newtheorem{theorem}{Theorem}
\newtheorem{lemma}{Lemma}

\newtheorem{definition}{Definition}

\newtheorem{remark}{Remark}

\def\ph{\varphi}
\def\RR{{\rm I\!\!R}}
\def\NN{{\rm I\!\!N}}

\def\<{\langle}
\def\>{\rangle}
\newcommand{\norm}[1]{\left\|#1\right\|}
\newcommand{\normtv}[1]{\|#1\|_{\mathrm{TV}}}
\newcommand{\pds}[2]{\<#1,#2\>}
\newcommand{\nnorm}[1]{\big|\!\big|\!\big|#1\big|\!\big|\!\big|}
\newcommand{\parenth}[1]{\left(#1\right)}
\newcommand{\crochets}[1]{\left[#1\right]}
\newcommand{\prox}{\mathrm{prox}}
\newcommand{\rprox}{\mathrm{rprox}}
\newcommand{\Id}{\mathrm{Id}}
\newcommand{\HT}{{{\cal T}_{\mathrm{H}}}}
\newcommand{\ST}{{{\cal T}_{\mathrm{S}}}}
\newcommand{\E}[1]{\mathbb{E}\crochets{#1}}
\newcommand{\var}[1]{\mathrm{Var}\crochets{#1}}
\def\phi{\varphi}

\def\sign{\mbox{\rm sign}}
\def\bpic{\begin{picture}}  \def\epic{\end{picture}}
\def\bfig{\begin{figure}}   \def\efig{\end{figure}}
\def\bc{\begin{center}}     \def\ec{\end{center}}
\def\beq{\begin{equation}}  \def\eeq{\end{equation}}
\def\beqn{\begin{eqnarray}} \def\eeqn{\end{eqnarray}}
\def\beqnn{\begin{eqnarray*}}   \def\eeqnn{\end{eqnarray*}}
\def\barr{\begin{array}}    \def\earr{\end{array}}
\def\bit{\begin{itemize}}   \def\eit{\end{itemize}}
\def\ben{\begin{enumerate}} \def\een{\end{enumerate}}

\def\h{\hat}    \def\t{\tilde}
\def\o{\overline}
\def\disp#1{{\displaystyle #1}}
\def\bm#1{\mbox{\boldmath $#1$}}

\def\C{{\cal C}}
\def\F{{\cal F}}
\def\H{{\cal H}}
\def\T{{\cal T}}

\def\X{{\cal X}}

\def\Wi{\widetilde W}
\def\W{W}
\def\w{w}
\def\wi{\widetilde w}
\def\O{\Omega}
\def\bc{\begin{center}}
\def\ec{\end{center}}
\def\UN{{\mathchoice {\rm 1\mskip-4mu l} {\rm 1\mskip-4mu l}{\rm 1\mskip-4.5mu l} {\rm 1\mskip-5mu l}}}
\def\pdf{{\rm pdf}}
\def\nablaD{\ddot{\nabla}}
\def\divD{{\rm Div}}
\def\gradf{\mathsf{D}}

\input{isolatin1.sty}

\begin{document}

\title{\Large\bf Multiplicative Noise Removal Using L1 Fidelity on Frame Coefficients}
\author{Durand S.$^*$, Fadili J.$^\dagger$ and Nikolova M.$^\diamond$}
\date{\small $^*$ M.A.P. 5, Université René Descartes (Paris V), 45 rue des Saint Pères, 75270 Paris Cedex 06, France\\
email: Sylvain.Durand@mi.parisdescartes.fr\\
$^\dagger$  GREYC CNRS-ENSICAEN-Universi\'e de Caen
6, Bd Maréchal Juin 14050 Caen Cedex, France\\
email: Jalal.Fadili@greyc.ensicaen.fr\\
$^\diamond$ CMLA, ENS Cachan, CNRS, PRES UniverSud,  61 Av. President Wilson, 94230 Cachan, France
  \\ email: Mila.Nikolova@cmla.ens-cachan.fr\\
  (\em alphabetical order of the authors)
}
\maketitle
\thispagestyle{empty}

\begin{abstract}
We address the denoising of images contaminated with multiplicative noise, e.g. speckle noise.
Classical ways to solve such problems are filtering, statistical (Bayesian) methods, variational methods, and
methods that convert the multiplicative noise into additive noise (using a logarithmic function),
shrinkage of the coefficients of the log-image data in a wavelet basis or in a frame, and transform back the result
using an exponential function.

We propose a method composed of several stages: we use the log-image data and apply a reasonable
under-optimal hard-thresholding on its
curvelet transform; then we apply a variational method where we minimize a specialized criterion composed
of an $\ell^1$ data-fitting to the thresholded coefficients and a Total Variation regularization (TV)
term in the image domain;
the restored image is an exponential of the obtained minimizer, weighted in a way that
the mean of the original image is preserved.
Our restored images combine the advantages of shrinkage and variational
methods and avoid their main drawbacks.
For the minimization stage, we propose a properly adapted  fast minimization scheme based on
Douglas-Rachford splitting. The existence of a minimizer of our specialized criterion being proven,
we demonstrate the convergence of the minimization scheme.
The obtained numerical results outperform the main alternative methods.
\end{abstract}

\addtolength{\baselineskip}{1.2mm}

\section{Introduction}

In various active imaging systems, such as synthetic aperture radar, laser or ultrasound imaging, the data
representing the underlying (unknown image) $S_0:\O\to\RR_+$, $\O\subset\RR^2$, are corrupted with multiplicative noise.
It is well known that such a noise severely degrades the image (see Fig.~\ref{diffT}(a)).
In order to increase the chance of restoring a cleaner image, several independent measurements for the same image are
realized, thus yielding a set of data:
\beq S_k =  S_0\,\eta_k+n_k,~~~\forall k\in\{1,\cdots,K\},\label{S_k}\eeq
where $\eta_k:\O\to\RR_+$, and $n_k$ represent the multiplicative and the additive noise relevant to each measurement $k$.
Usually, $n_k$ is white Gaussian noise.
A commonly used and realistic model for the distribution of $\eta_k$ is the one-sided exponential distribution:
\[\eta_k ~~: ~~ \pdf(\eta_k) = \mu\, e^{-\mu \eta_k}\,\UN_{\RR_+}(\eta_k);\]
the latter is plotted in Fig. \ref{LAWS}(a).
Let us remind that $1/\mu$ is both the mean and the standard deviation of this distribution.
The usual practice is to take an average of the set of all measurements---such an image can be seen in (see Fig.~\ref{diffT}(b)).
Noticing that
 $\disp{\frac{1}{K}\sum_{k=1}^Kn_k\approx0}$, the data production model reads
\beq S = \frac{1}{K}\sum_{k=1}^K S_k =  S_0\,\frac{1}{K}\sum_{k=1}^K\eta_k = S_0\,\eta,\label{expo}\eeq
see e.g. \cite{Ulaby89,Achim01,Xie02} and many other references.
A reasonable assumption is that all $\eta_k$ are independent and share
the same mean $\mu$.
Then the resultant mean of the multiplicative noise $\eta$ in \eqref{expo} is known to follow a
Gamma distribution,
\beq \eta =  \frac{1}{K}\sum_{k=1}^K\eta_k : ~~~ \pdf(\eta) =
\disp{\left(\frac{K}{\mu}\right)^K\frac{\eta^{K-1}}{\Gamma(K)}\exp\left(-\frac{K\eta}{\mu}\right)},\label{Gamma}
\eeq
where $\Gamma$ is the usual Gamma-function and since $K$ is integer, $\Gamma(K)=(K-1)!$.  Its mean is again $\mu$ and its standard deviation is $\mu/K$.
It is shown in Fig. \ref{LAWS}(b).

Various adaptive filters for the restoration of images contaminated with multiplicative noise
have been proposed in the past, e.g. see \cite{Yu02,Krissian07} and the numerous references therein.
It can already been seen that filtering methods work well basically when the noise is moderate or weak, i.e. when $K$ is large.
Bayesian or variational methods have been proposed as well; one can consult for instance
\cite{Walessa00,Rudin03,AubertAujol08,Ng08} and the references cited therein.

A large variety of methods---see e.g. \cite{Fukuda98,Achim06}, more references are given in
\S~\ref{logdatamethods}---rely on
the conversion of the multiplicative noise into additive noise using
\beq v = \log S = \log S_0 + \log \eta = u_0 + n. \label{logN}\eeq
In this case the probability density function  of $n$ reads (see Fig. \ref{LAWS}(c)):
\beq
n = \log \eta : ~~~\pdf(n) = \disp{\left(\frac{K}{\mu}\right)^K\frac{1}{\Gamma(K)}\exp K(n-\mu e^n)} ~.
\label{3gamma}\eeq
One can prove that
\beqn
\E{n}&=&\psi_0(K)-\log K\label{esp} ~,\\
\var{n} &=& \psi_1(K)\label{var},
\eeqn
 where
\beq\psi_k(z) = \parenth{\frac{d}{dz}}^{k+1}\log\Gamma(z)\label{polyG}\eeq
is the polygamma function \cite{Abramovitz72}.

Classical SAR modeling---see \cite{Tur82,Ulaby89} and many other references---correspond to $\mu=1$ in \eqref{Gamma}.
Then \eqref{Gamma} and \eqref{3gamma} boil  down to
\beqn
\pdf(\eta) &=&
\frac{K^K\eta^{K-1}e^{-K\eta}}{(K-1)!} ~,\nonumber\\
\pdf(n) &=& \frac{K^Ke^{K(n-e^n)}}{(K-1)!} ~.
\label{pdfn}
\eeqn

\bfig
\begin{center}
\begin{tabular}{cccc}
\epsfig{figure=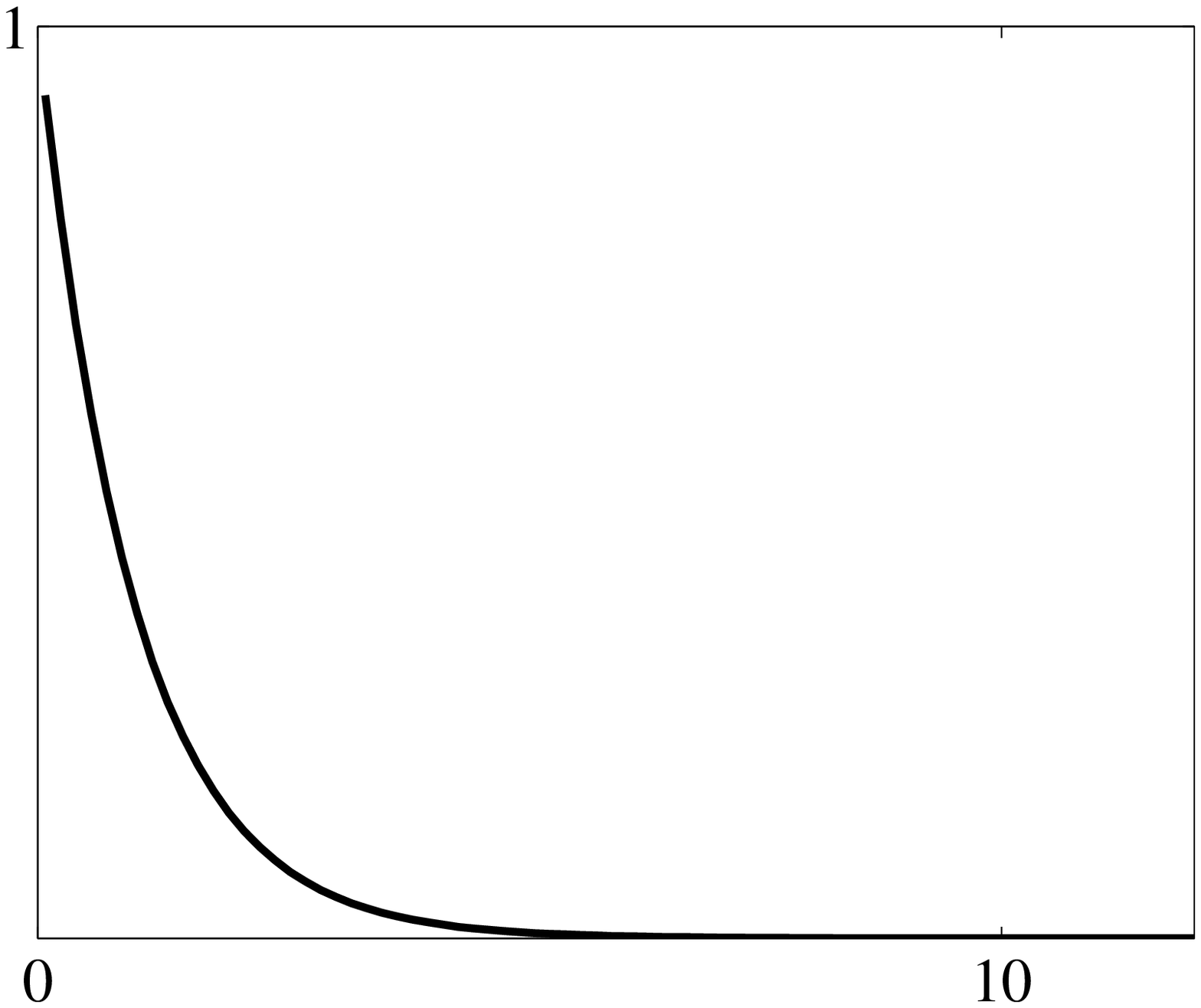,width=3.4cm}&
\epsfig{figure=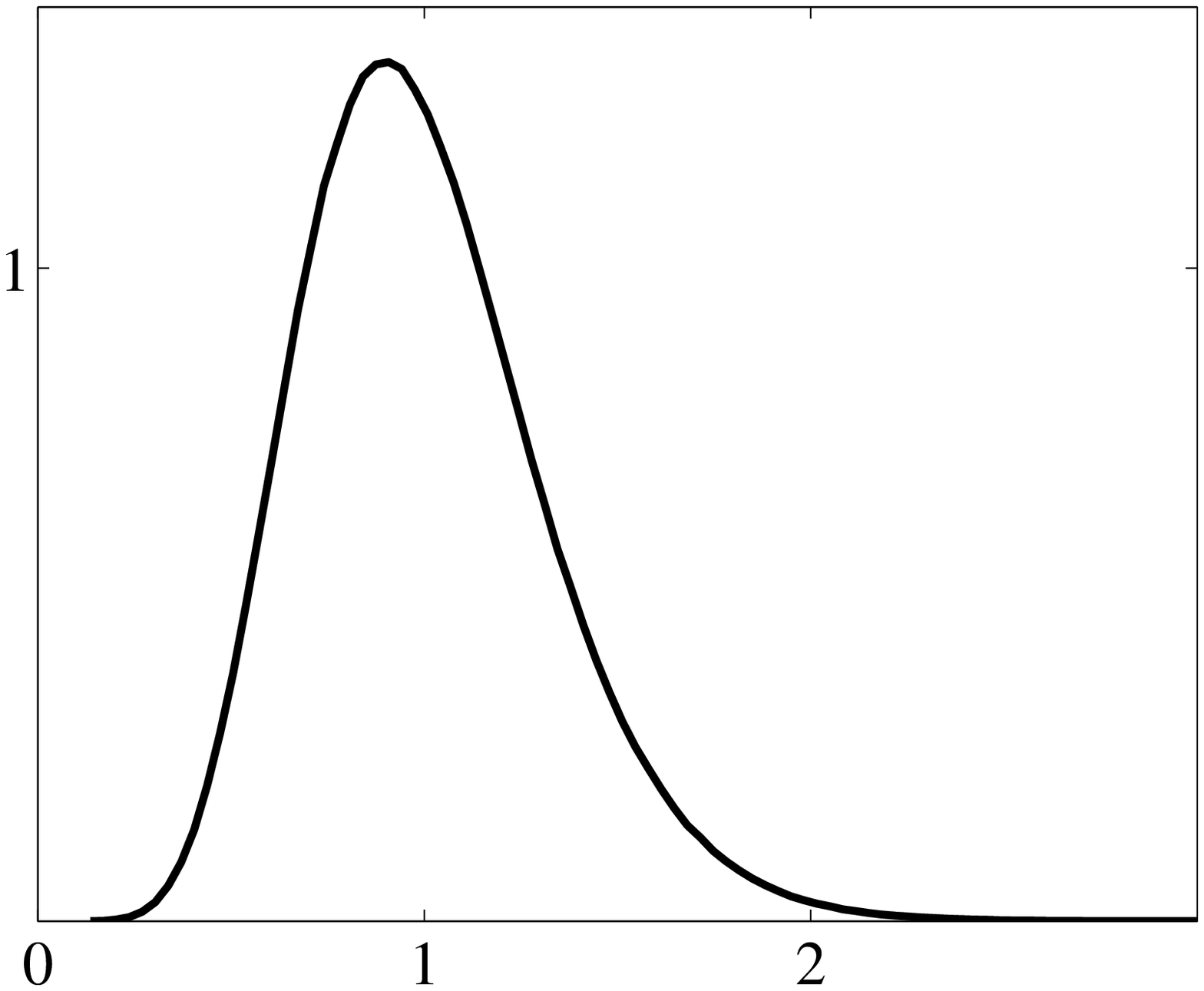,width=3.4cm}&
\epsfig{figure=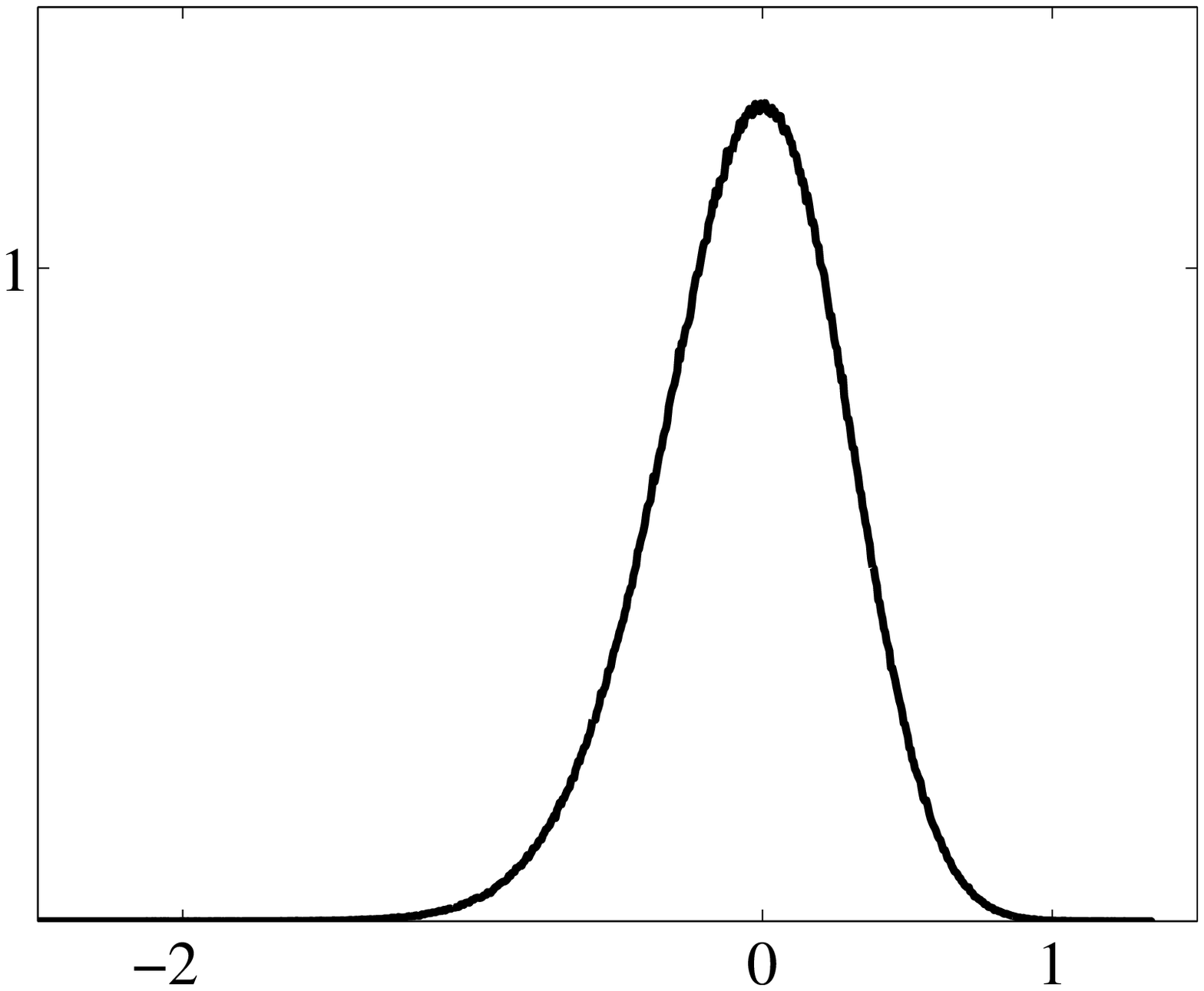,width=3.4cm}&\epsfig{figure=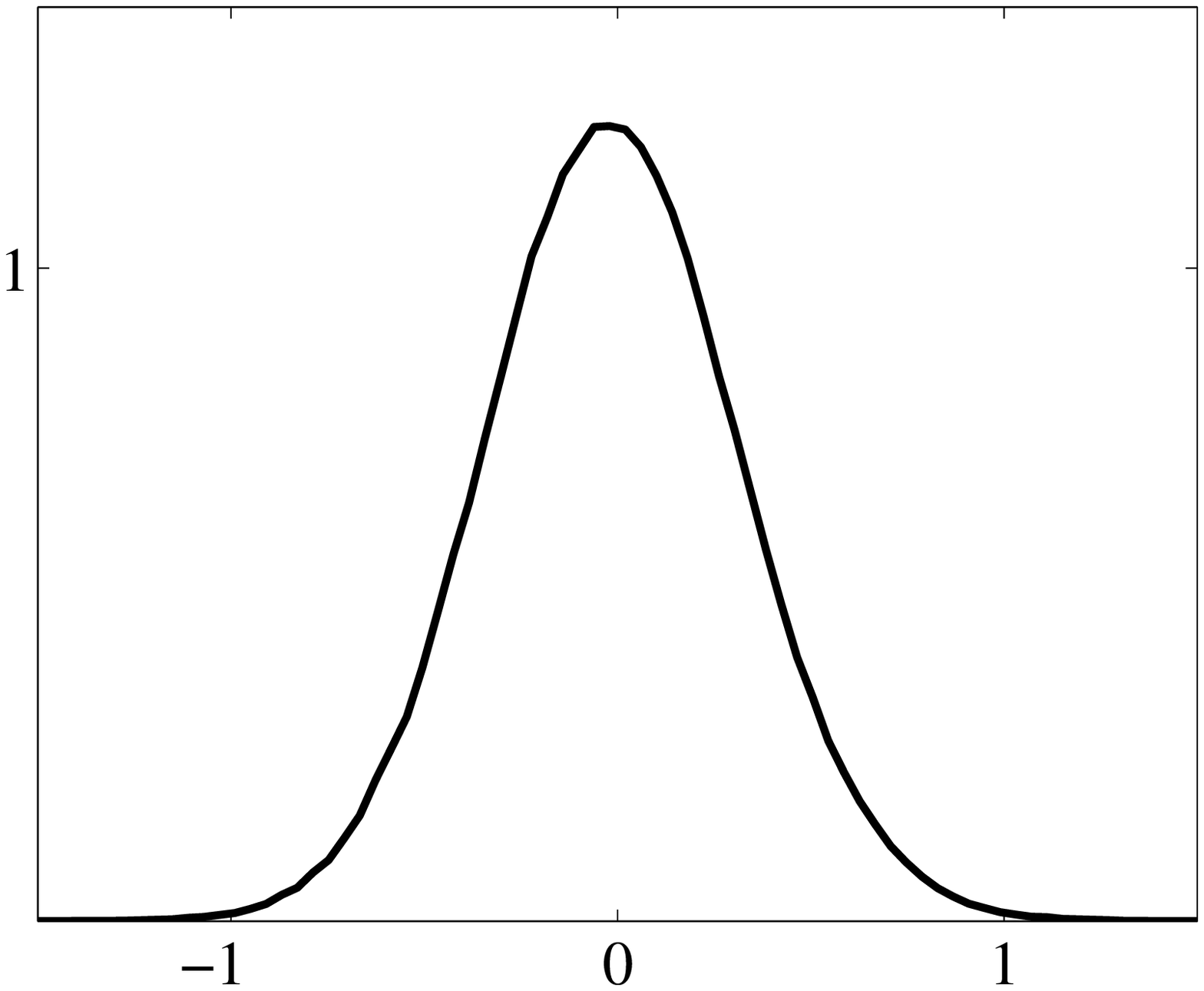,width=3.4cm} \\
{(a)~ $\eta_k$}&
{(b) ~$\eta = \frac{1}{K}\sum_{k=1}^K\eta_k$}&
{(c)~ $n = \log \eta$}&
{(d) ~$Wn$}
\end{tabular}
\end{center}
\caption{Noise distributions.}
\label{LAWS}
\efig

\subsection{Multiscale shrinkage for the log-data}\label{logdatamethods}

Many authors---see \cite{Fukuda98,Xie02,Achim01,Achim03,Pizurica06} and the references given there---focus
on restoring the log-data as given in  \eqref{logN}.
The common strategy is to decompose the log-data into some multiscale frame for $L^2(\RR^2)$, say
$\{\langle\w_i\rangle :i\in I\}$:
\beq y=\W v=\W u_0+\W n,\label{Wv}\eeq
where $\W$ is the corresponding frame  analysis operator, i.e. $(\W v)[i]=\langle v,\w_i\rangle$, $\forall i\in I$.
The rationale is that the noise $\W n$ in $y$ is
nearly Gaussian---as seen in Fig. \ref{LAWS}(d)---and justified by the Central Limit Theorem.
The obtained coefficients $y$ have been considered in different frameworks in the literature.
In a general way, coefficients are restored using shrinkage estimators using a symmetric
function $\T:\RR\to\RR$, thus yielding
\beq y_\T[i]=\T\big((\W v)[i]\big),~~~\forall i\in I.\label{aze1}\eeq
Following \cite{Donoho94}, various shrinkage estimators $\T$ have been explored in the literature,
\cite{Donoho95a,Wang95,Simoncelli96,Moulin99,Belge00,Antoniadis01};
see \S~\ref{Drawbacks} for more details on shrinkage methods.
 Shrinkage functions specially designed for multiplicative noise were proposed e.g. in \cite{Achim01,Xie02,Achim03}.

Let $\Wi$ be a left inverse of $\W$, giving rise to the dual frame $\{\wi_i:i\in I\}$.
Then a denoised log-image $v_\T$ is generated by  expanding the shrunk coefficients $y_\T$
in the dual frame:
\beq\label{vt} v_\T  =\sum_{i\in I} \T((\W v)[i])\,\wi_i
=\sum_{i\in I} \T(y[i])\,\wi_i.\eeq
Then the sought-after image is of the form $S_\T = \exp v_\T.$

\subsection{Our approach and organization of the paper}

We first apply \eqref{logN} and then consider a tight-frame transform of the log-data.
Our method to restore the log-image is presented in section \ref{log-image}.
It is based on the minimization of a criterion composed of an $\ell^1$-fitting to the (suboptimally)
hard-thresholded frame coefficients and a Total Variation (TV) regularization in the image domain.
This method uses some ideas from a previous work of some of the authors \cite{Durand07}.
The minimization scheme to compute the log-restored image, explained in  section \ref{minsch}, uses a
Douglas-Rachford  splitting specially adapted to our criterion.
Restoring the sought-after image from the restored log-image requires a bias correction which is
presented in section~\ref{resto-expo}.
The resultant algorithm to remove the multiplicative noise is provided  in section  \ref{DA}.
Various experiments are presented in section \ref{Exper}.
Concluding remarks are given in section~\ref{conclusion}.

\section{Restoration of the log-image}\label{log-image}

In this section we consider how to restore a good log-image given data $v:\Omega\to\RR$ obtained
according to~\eqref{logN}.
We focus basically on methods which, for a given preprocessed data set, lead to convex optimization problems.
Below we comment only variational methods and shrinkage estimators
since they underly the method proposed in this paper.

\subsection{Drawbacks of shrinkage restoration and variational methods}\label{Drawbacks}
\paragraph{Shrinkage restoration.}The major problems with shrinkage denoising methods, as sketched in
\eqref{aze1}-\eqref{vt}, is that shrinking large coefficients
entails an erosion of the spiky image features, while shrinking small coefficients towards zero
yields Gibbs-like oscillations in the vicinity of edges and a loss of texture information.
On the other hand, if shrinkage is not sufficiently strong, some coefficients bearing mainly
noise will remain almost unchanged---we call such coefficients {\em outliers}---and (\ref{vt})
suggests they generate artifacts with the shape of the functions $\wi_i$ of the frame.
A well instructive illustration can be seen in Fig.~\ref{diffT}(b-h).
Several improvements, such as translation invariant thresholding \cite{Coifman95} and block
thresholding \cite{Chesneau08}, were brought to shrinkage methods in order to alleviate these artifacts.
Results obtained using the latter method are presented in Figs.~\ref{fig:phantom}(c), \ref{fig:lena}(d) and
\ref{fig:boat}(d) in Section~\ref{Exper}.
Another  inherent difficulty comes from the fact that coefficients between different scales are
not independent, as usually assumed, see e.g. \cite{Simoncelli99,Moulin99,Belge00,Antoniadis02}.

\paragraph{Variational methods.}
In variational methods, the restored function is defined as the minimizer of a criterion
$\F_v$ which balances trade-off between closeness to data and regularity constraints,
\beq\F_v(u)= \rho\int_\O\psi\big(u(t),v(t)\big)dt+
\int_{\O}\phi(|\nabla u(t)|)\,dt,\label{Fr}\eeq
where $\psi:\RR_+\to\RR_+$ helps to measure closeness to data,
$\nabla$ stands for gradient (possibly in a distributional sense),
$\phi:\RR_+\to\RR_+$ is called a potential function and $\rho>0$ is a parameter.
A classical choice for $\psi$ is $\psi(u(t),v(t))=\big(u(t)-v(t)\big)^2$ which assumes that the noise
$n$ in \eqref{logN} is white, Gaussian and centered.
Given the actual distribution of the noise in \eqref{pdfn} and Fig. \ref{LAWS}(c),
this may seem hazardous; we reconsider this choice in \eqref{L2TV}.
A reasonable choice is to use the log-likelihood of $n$ according to \eqref{pdfn} and this
was involved in the criterion proposed in \cite{Ng08}---see \eqref{NgCrit} at the end of this paragraph.

Let us come to the potential function $\phi$ in the regularization term.
In their pioneering work, Tikhonov and Arsenin \cite{Tikhonov77} considered $\phi(t)=t^2$; however it
is well known that this choice for $\phi$ leads to smooth images with flattened edges.
Based on a fine analysis of the minimizers of $\F_v$ as solutions of PDE's on
$\Omega$, Rudin, Osher and Fatemi \cite{Rudin92} exhibited that $\phi(|\nabla u(t)|)=\|\nabla u(t)\|_2$,
where $\|.\|_2$ is the $L^2$-norm, leads to images
involving edges. The resultant regularization term is known as Total Variation (TV).
However, whatever smooth data-fitting is chosen, this regularization yields images
containing numerous constant regions (the well known stair-casing effect),
so that textures and fine details are removed, see \cite{Nikolova00b}.
The method in \cite{AubertAujol08} is of this kind and operates only  on the image domain; the fitting term is derived from
\eqref{Gamma} and the criterion reads
\beq \F_S(\Sigma)=\rho\int \left(\log \Sigma(t)+\frac{S(t)}{\Sigma(t)}\right)dt+\normtv{\Sigma},\label{AA}\eeq
where $\rho$ depends on $K$. The denoised image $\disp{ \h S_0 =\arg\min_\Sigma\F_S}$ exhibit constant regions, as seen in
Figs.~\ref{fig:lena}(e) and \ref{fig:boat}(e) in Section~\ref{Exper}.
We also tried to first restore the log-image $\h u$ by minimizing
\beq
F_v(u)=\rho\|u-v\|^2+\|u\|_{\rm TV}
\label{L2TV}
\eeq
and the sought after image is of the form $\h S_0=B \exp(\h u)$ where $B$ stands for the bias correction explained in
section \ref{resto-expo}. Because of the exponential transform, there is no stair-casing, but some outliers remain visible---see
Figs. \ref{fig:lena}(c) and \ref{fig:boat}(c); nevertheless, the overall result is very reasonable.
The result of  \cite{Rudin92} was at the origin of a large amount of papers
dedicated to constructing edge-preserving convex potential functions, see e.g. \cite{Acar94,Charbonnier97,Vogel96},
and for a recent overview, \cite{Aubert06}. Even though smoothness at the origin  alleviates
stair-casing, a systematic drawback of the images restored using all these functions
$\phi$ is that the amplitude of edges is underestimated---see e.g. \cite{Nikolova04a}.
This is particularly annoying if the sought-after function has neat edges or spiky areas since the later are
subjected to erosion. A very recent method proposed in \cite{Ng08} restores the discrete log-image
using the log-likelihood of \eqref{pdfn} and a regularized TV; more precisely,
\beq \F_v(u,w)=\sum_{i}\left(u[i]+S[i]e^{-u[i]}\right)+\rho_0\|u-w\|^2+\rho\normtv{w}\label{NgCrit},\eeq
where the denoised log-image $\h u$ is obtained using alternate minimization on $u$ and $w$.
The TV term here is regularized via $\|u-w\|^2_2$ and the resultant denoised image is given by $\h S_0=\exp(\h w)$.
The results present some improvement with respect to the method proposed in  \cite{AubertAujol08},
at the expense of two regularization parameters ($\rho$ and $\rho_0$) and twho stopping rules for each one of
the minimization steps.

\subsection{Hybrid methods} \label{hybrid}

Hybrid methods \cite{Bijaoui97,Coifman00,Chan00p,Froment01,Malgouyres02,Malgouyres02a,Candes02,Durand03a}
combine the information contained in the large coefficients $y[i]$, obtained according to
\eqref{Wv}, with pertinent priors directly on the log-image $u$.

\begin{remark} Such a framework is particularly favorable in our case since the noise $Wn[i],~i\in I$
in the coefficients
$y[i],~i\in I$, have a nearly Gaussian distribution---see Fig.~\ref{LAWS}(d).
\end{remark}

Although based on different motives, hybrid methods amount to define the restored
function $\h u$ as
\[ \mbox{minimize}~~\Phi(u)\]
\[\mbox{subject to}~~
\h u\in \left\{u:\;\left|\left(\W(u-v)\right)[i]\right|\leq\mu_i,\,\forall i\in I\right\}.\]
If the use of an edge-preserving regularization, such as TV for $\Phi$ is a pertinent choice,
the strategy for the selection of parameters $\{\mu_i\}_{i\in J}$ is
more tricky.
This choice must take into account
the magnitude of the relevant data coefficient $y[i]$.
However, deciding on the value of $\mu_i$ based solely on $y[i]$, as done in these papers,
is too rigid since there are either correct data coefficients that incur smoothing ($\mu_i>0$),
or noisy coefficients that are left unchanged ($\mu_i=0$).
A way to alleviate this situation is to determine $(\mu_i)_{i\in I}$ based both on
the data and on a prior regularization term.
Following \cite{Nikolova01a,Nikolova03a}, this objective is carried
out by defining restored coefficients $\h x$ to minimize the non-smooth objective function, as explained below.

\subsection{A specialized hybrid criterion}\label{proposed}
Given the log-data $v$ obtained according to \eqref{logN}, we first apply a frame transform as in \eqref{Wv} to get
$y=\W v=\W u_0+\W n$.
We systematically denote by $\h x$ the denoised coefficients.
The noise contained in the $i$-th datum reads $\<n,\w_i\>$ whose distribution is
displayed in Fig.~\ref{LAWS}(d).
The low frequency approximation coefficients carry important information about the image.
In other words, when $\w_i$ is low frequency, then $\<n,\w_i\>$ has a better SNR than other coefficients.
Therefore, as usual, a good choice is to keep them intact at this preprocessing stage.
Let  $I_*\subset I$ denote the subset of all such elements of the frame.
Then we apply a hard-thresholding to all coefficients except those contained in $I_*$
\beq y_\HT[i]\defeq\HT\big(y[i]\big),~~~\forall i\in I\setminus I_*,\label{aze}\eeq
where the  hard-thresholding operator $\HT$ reads \cite{Donoho94}
\beq\label{hard}\HT(t)=\left\{\barr{ll}0&\mbox{if}~~|t|\leq
T,\\t&\mbox{otherwise.}\earr\right.\eeq
The resultant set of coefficients is systematically denoted by $y_\HT$.
We choose an {\sl underoptimal} threshold $T$ in order to preserve
as much as possible the information relevant to edges and to textures, an important part of which is
contained in the small coefficients.
Let us consider
\beq v_\HT  =\sum_{i\in I_1} \W v[i]\,\wi_i, \label{bze}\eeq
where
\beq I_1=\{i\in I\setminus I_*:\,|y[i]|> T\}\label{I1}.\eeq
The image $v_\HT$ contains a lot of artifacts with the shape of the $\wi_i$ for those $y[i]$ that
are noisy but above the threshold $T$, as well as a lot of information about the fine details in the
original log-image $u_0$. In all cases, whatever the choice of $T$, the image of the form
$v_\HT$ is unsatisfactory---see Fig.~\ref{diffT} (b-h).

\begin{figure}
\begin{tabular}{cccc}
\includegraphics[width=0.23\textwidth]{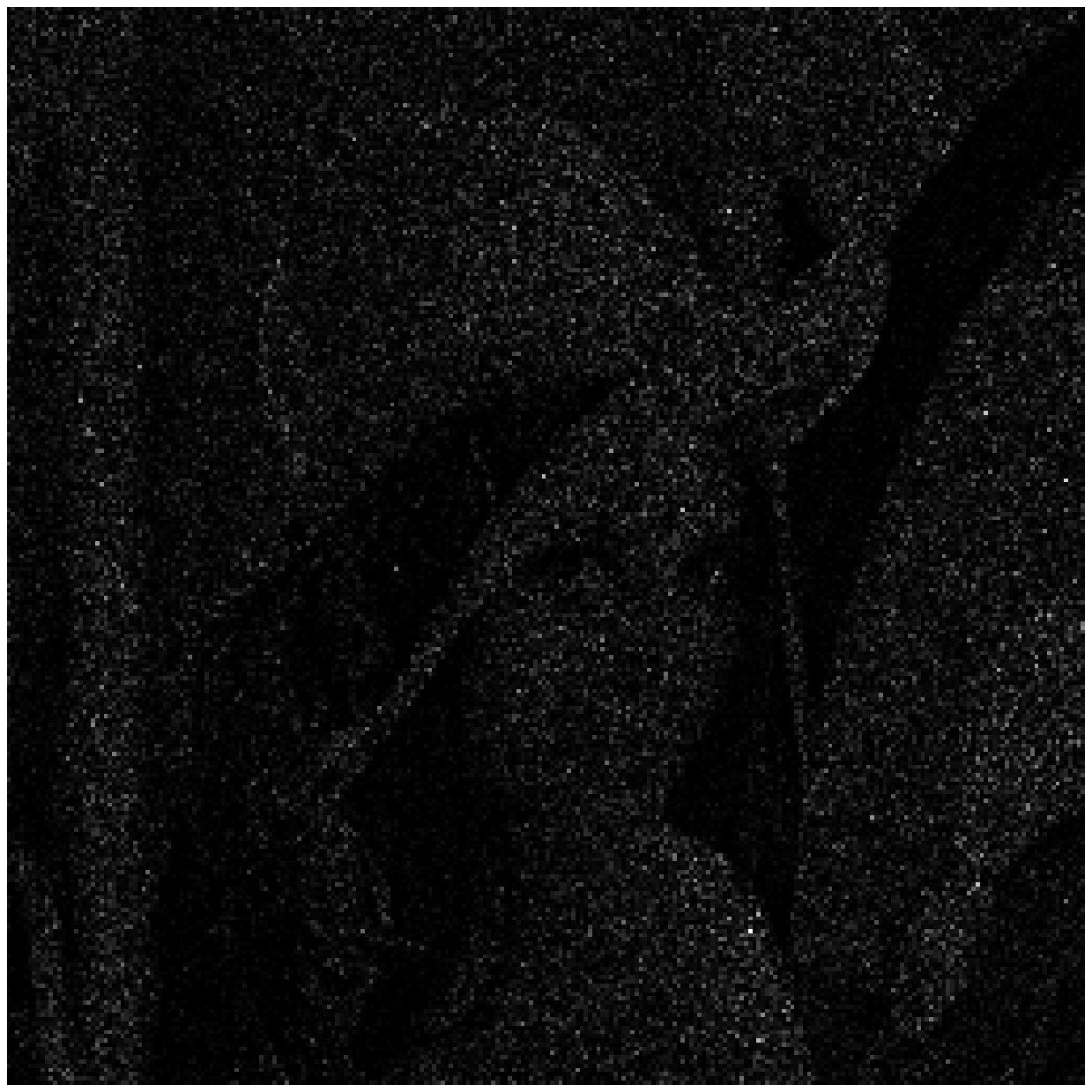}&
\includegraphics[width=0.23\textwidth]{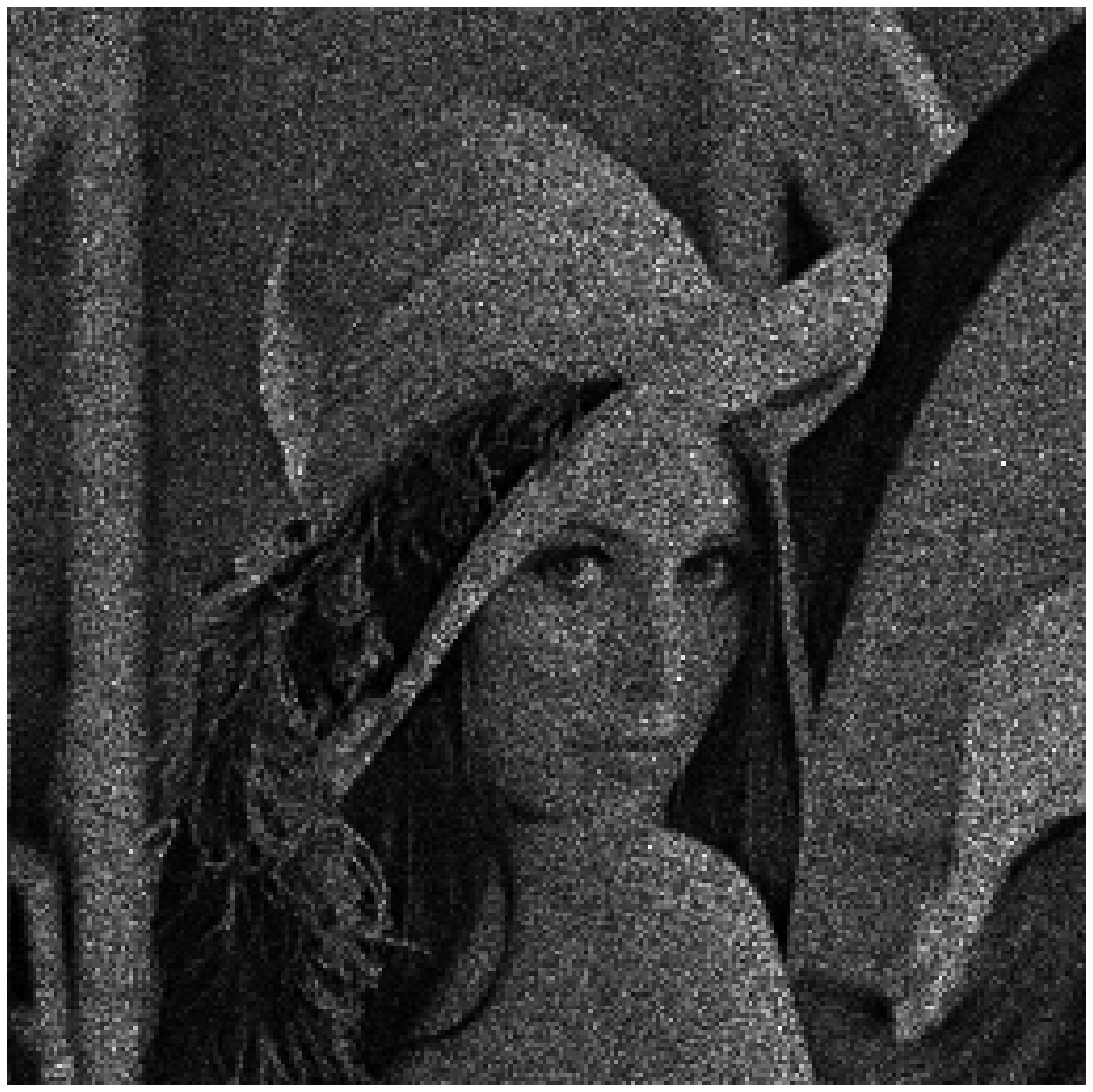}&
\includegraphics[width=0.23\textwidth]{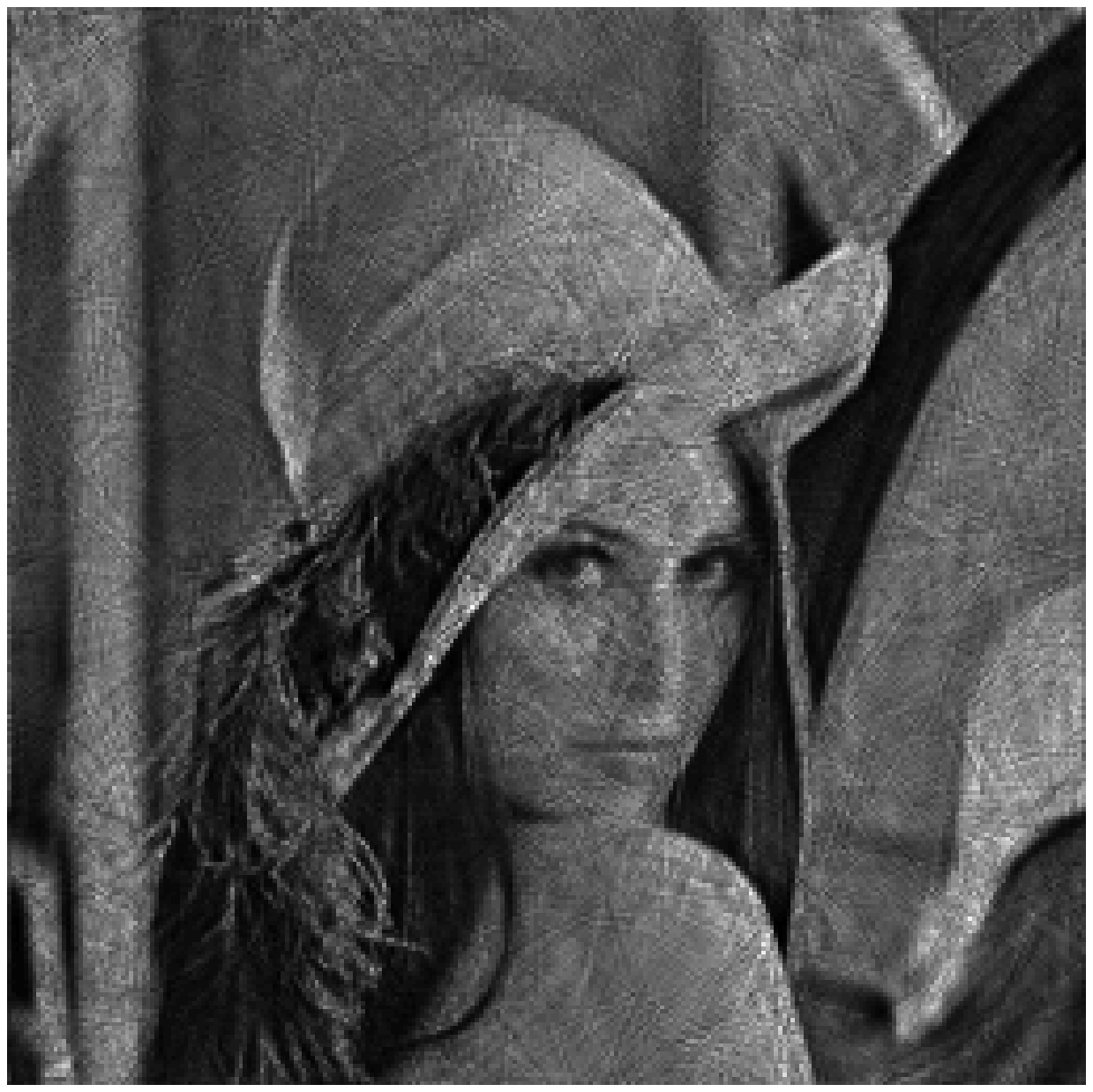}&
\includegraphics[width=0.23\textwidth]{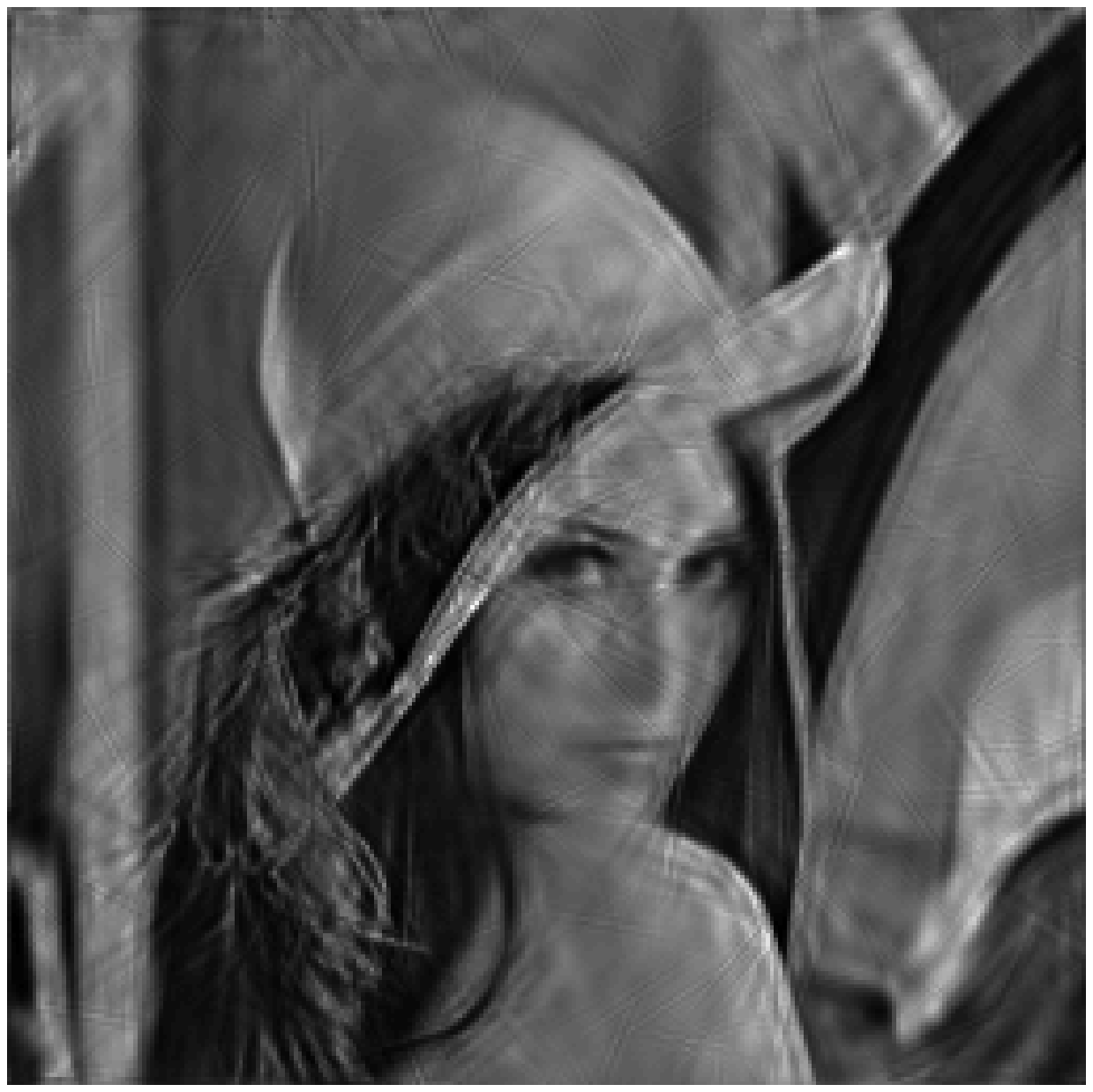}\\
{(a)~Noisy for $K=1$ }&{(b) Noisy for $K=10$}&{(c) ~$T=2\sigma$} & {(d) ~$T=3\sigma$}\\
\includegraphics[width=0.23\textwidth]{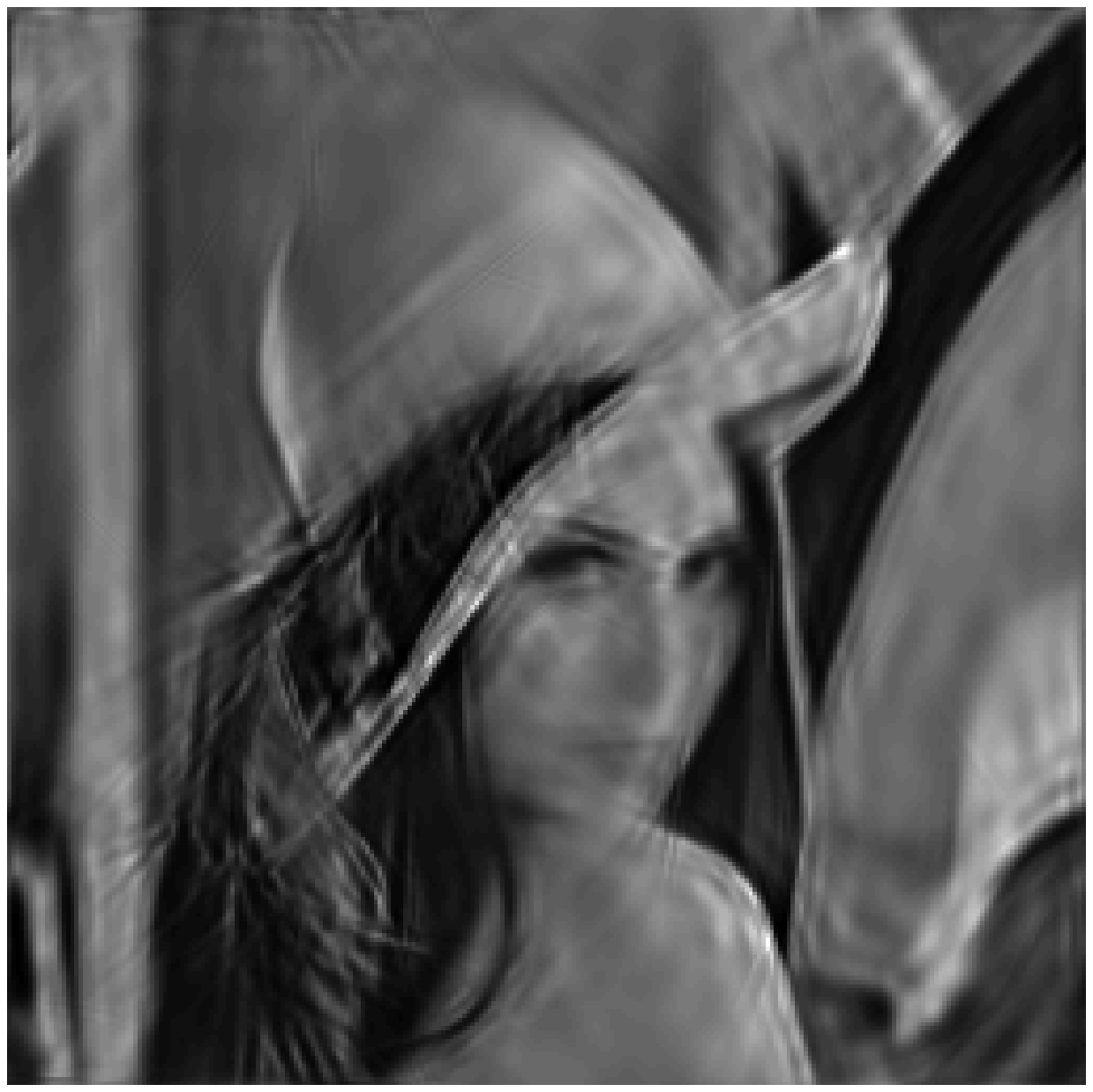}&
\includegraphics[width=0.23\textwidth]{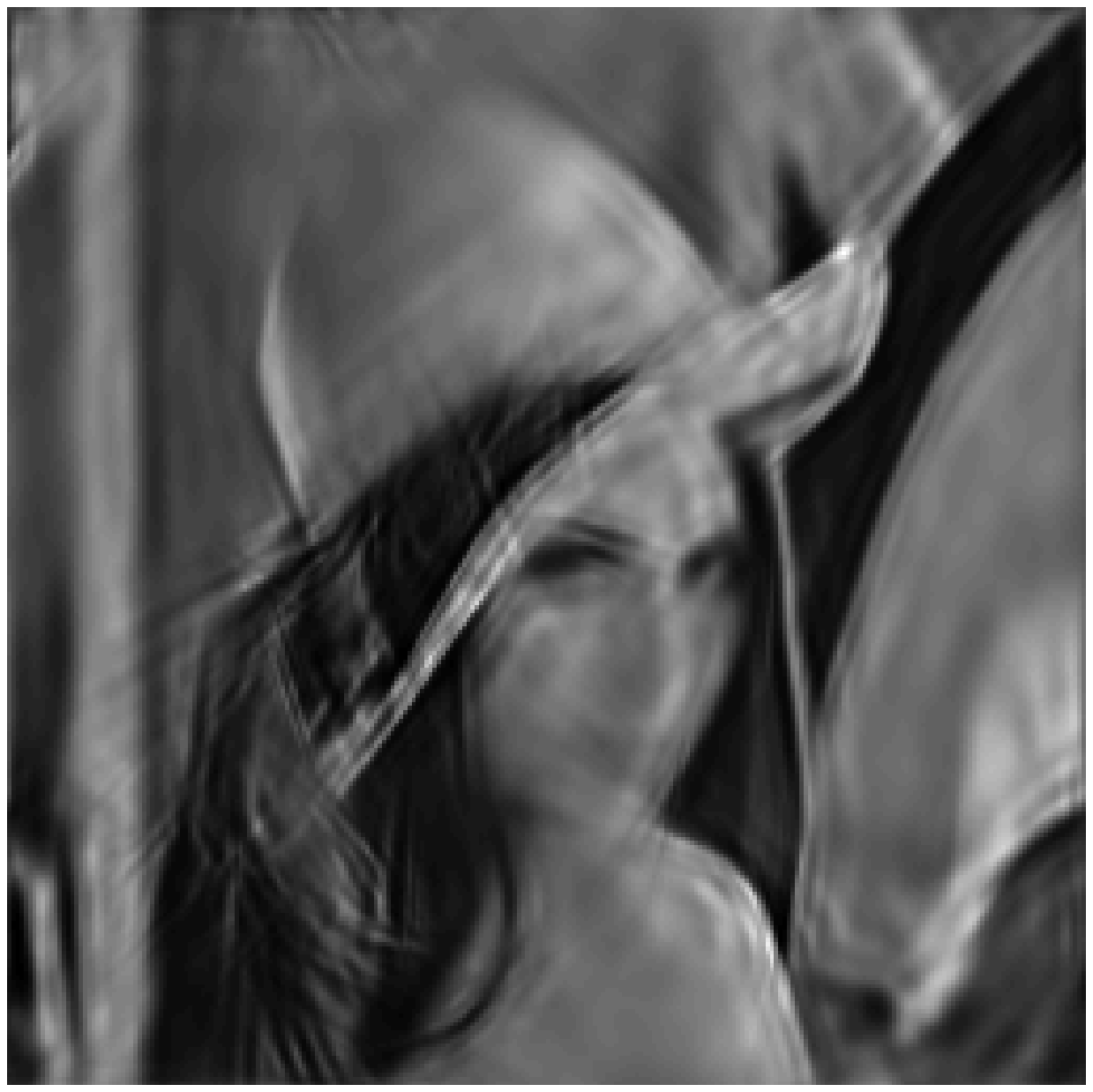}&
\includegraphics[width=0.23\textwidth]{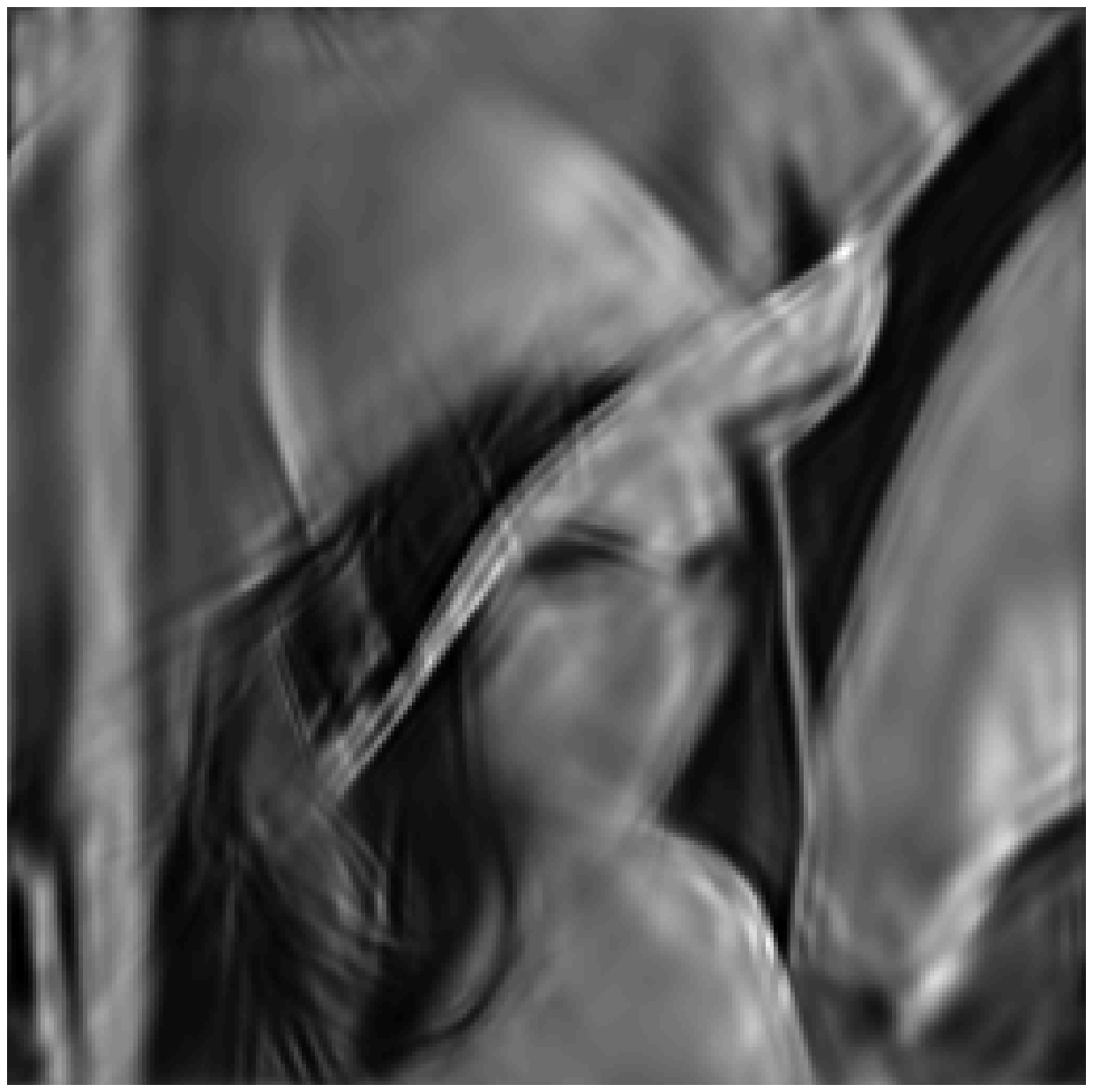}&
\includegraphics[width=0.23\textwidth]{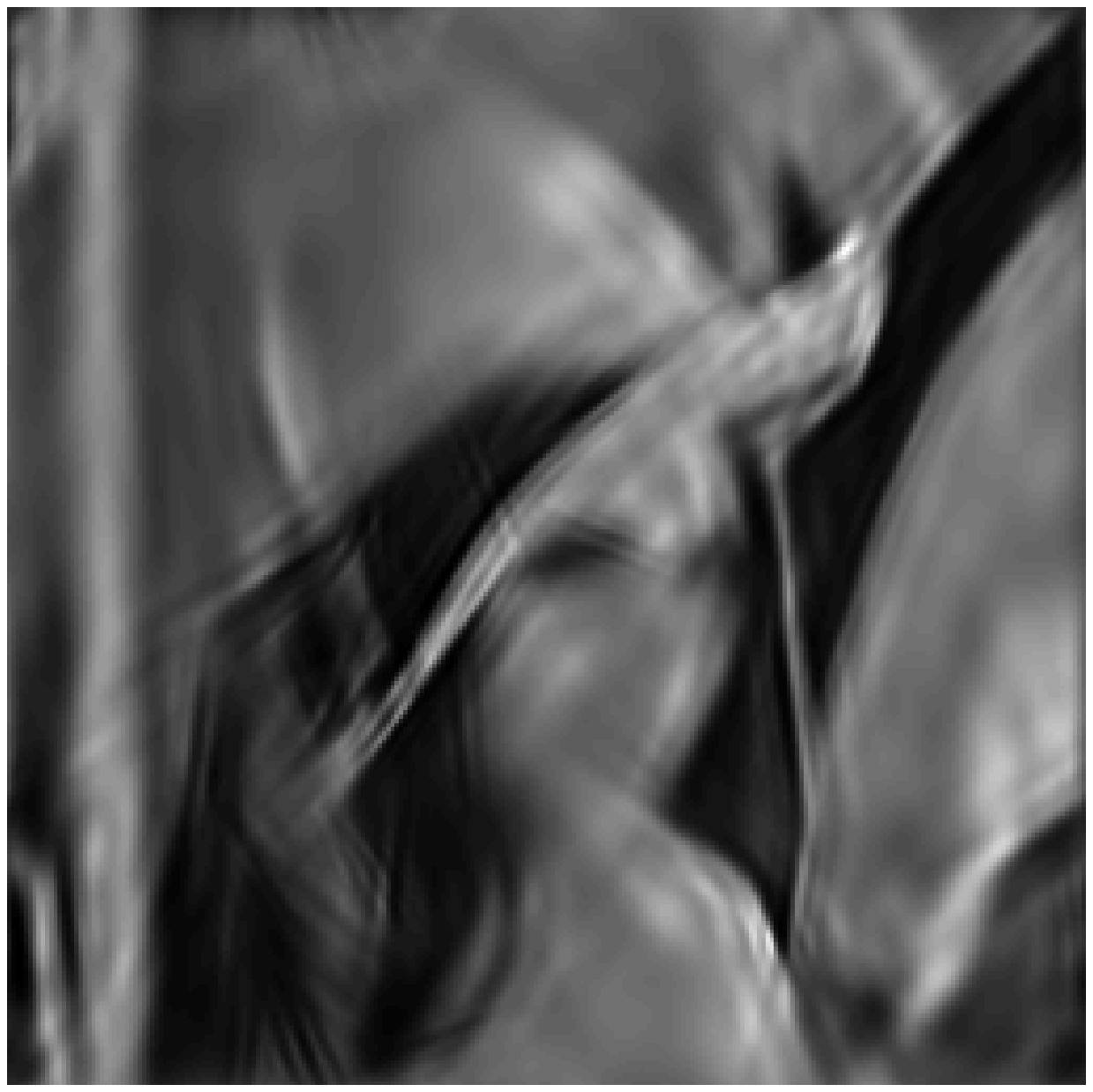}\\
{(e) ~$T=4\sigma$}&{(f) ~$T=5\sigma$}&{(g) ~$T=6\sigma$}&{(h) ~$T=8\sigma$} \\
\end{tabular}
\caption{(a) Noisy Lena for $K=1$. (b) Noisy Lena obtained via averaging, see \eqref{S_k}, for $K=10$.
(c)-(h) Denoising of data $v$ shown in (c) where
the curvelet trasform of $v$ are hard-thresholded
according to \eqref{aze}-\eqref{bze} for different choices of $T$ where
$\sigma=\sqrt{\psi_1(K)}$ is the standard deviation of the noise $n$. The displayed restorations correspond to $\exp{v_\HT}$,
see \eqref{bze}.}
\label{diffT}
\end{figure}

We will restore $\h x$ based on the under-thresholded data $y_\HT$.
We focus on hybrid methods of the form:
\beq \left\{ \barr{lll}\h x&=&\disp{\arg\min_x F(x)}\\~\\
F(x)&=&\disp{\Psi(x,y_\HT)+\Phi(\Wi x),}
\earr\right.
\label{FX}
\eeq
where $\Psi$ is a data-fitting term in the domain of the frame coefficients and $\Phi$
is an edge-preserving regularization term bearing the prior on the sought-after log-image $\h u$.
The latter sought-after log-image  $\h u$ reads
\beq\h u=\Wi \h x~.\label{logimagefine}\eeq

Next we analyze the information content of the coefficients $y_\HT$ that give rise to our log-image $\h u$.
Let us denote
\beq I_0=I\setminus\big(I_1\cup I_* \big)=\{i\in I\setminus I_*:\,|y[i]|\leq T\}.
\label{I0}\eeq
We are mostly interested by the information borne by the coefficients relevant to $I_0$ and $I_1$.
\ben
\item[($I_0$)]
The coefficients $y[i]$ for $i\in I_0$ are usually high-frequency components
which can be of the two types described below.
\ben
\item Coefficients $y[i]$ containing essentially noise---in which case the best we can do is to
keep them null, i.e. $\h x[i]=y[i]$;

\item Coefficients $y[i]$ which correspond to edges and other details in $u_0$.
Since $y[i]$ is difficult to distinguish from the noise, the relevant $\h x[i]$ should be restored
using the edge-preserving prior conveyed by $\Phi$.
Notice that a careful restoration must find a nonzero $\h x[i]$, since otherwise
$\h x[i]=0$ would generate Gibbs-like oscillations in $\h u$.
\een

\item[($I_1$)] The coefficients $y[i]$ for
$i\in I_1$ are of the following two types:
\ben
\item Large coefficients which carry the main features of the sought-after function.
They verify $y[i]\approx\<\w_i,u_0\>$ and can be kept intact.
\item Coefficients which are highly contaminated by noise, characterized by $|y[i]|\gg|\<\w_i,u_0\>|$.
We call them {\em outliers} because if we had $\h x[i]=y[i]$, then $\h u$ would contain an artifact
with the shape of $\wi_i$ since by \eqref{bze} we get
 $v_\HT=\sum_{j\setminus i}\h x[j]\wi_j+y[i]\wi_i$.
Instead, $\h x[i]$ must be restored according to the prior  $\Phi$.
\een
\een
This analysis clearly defines the goals that the minimizer $\h x$ of $F_y$ is expected to achieve.
In particular, $\h x$ must involve an implicit classification between coefficients that  fit
to $y_\HT$ exactly and coefficients that are  restored according to the prior term $\Phi$.
In short, restored coefficients have to fit $y_\HT$ exactly if
they are in accordance with the regularization term $\Phi$ and have to be restored via the later term otherwise.
Since \cite{Nikolova01a,Nikolova03a} we know that criteria $F_y$ where $\Psi$ is non-smooth
at the origin (e.g. $\ell^1$) can satisfy $\h x[i]=y_\HT[i]$ for coefficients that are in accordance
with the prior $\Phi$, while the others coefficients are restored according to $\Phi$, see also \cite{Durand07}.
For these reasons, we focus on a criterion on the form
\beq\label{FoF}
F_y(x)= \Psi(x) + \Phi(x) \eeq
where
\beqn\Psi(x)&=& \sum_{i\in I }\lambda_i\big|(x-y_\HT)[i]\big|~~=~~
    \sum_{i\in I_1\cup I_*}\lambda_i\left|(x-y)[i]\right|+\sum_{i\in I_0}\lambda_i\left|x[i]\right|,\label{Psio}\\
\Phi(x)&=& \int_\O|\nabla\Wi x|\,ds ~~=~~ \|\Wi x\|_{\mbox{\footnotesize TV}}.\label{Phio}
\eeqn

Note that in \eqref{FoF}, as well as in what follows, we write $F_y$ in place of $F_{y_\HT}$ in order to simplify the notations.

In the pre-processing step \eqref{hard} we would not recommend the use of a shrinkage function other than $\HT$
since it will alter {\em all} the data coefficients $y_\T$, without restoring them faithfully.
In contrast, we base our restoration on data $y_\HT$ where all non-thresholded coefficients keep
the original information on the sought-after image.

The theorem stated next ensures the existence of a minimizer for $F_y$ as defined in \eqref{FoF} and
\eqref{Psio}-\eqref{Phio}.
Its proof can be found in \cite{Durand07}.

\begin{theorem}{\rm\cite{Durand07}}\label{exist}
For $y\in\ell^2(I)$ and $T>0$ given, consider $F_y$ as defined in \eqref{FoF},
where $\O\in\RR^2$ is open, bounded and its boundary $\partial\Omega$ is Lipschitz.
Suppose that
    \ben

\item $\{\w_i\}_{i\in I}$ is a frame of $L^2(\O)$ and the operator $\Wi$ is the pseudo-inverse of $\W$; \label{F1}
    \item $\disp{\lambda_{\min}=\min_{i\in I}\lambda_i>0}$. \label{F3}
    \een
Then $F_y$ has a minimizer in $\ell^2(I)$.
\end{theorem}

Let us remind that the minimizer of $F_y$ is not necessarily unique.
Given $y$, denote
\beq G_y\defeq\big\{\h x\in \ell^2(I)~:~ F_y(\h x)=\min_{x\in\ell^2(I)} F_y(x)\big\}.\label{G}\eeq
Hopefully, for every sample of the preprocessed data $y_\HT$, the set $G_y$ is convex and corresponds
to images $\Wi\h x$ which are very similar since they share the same level lines.
The theorem below confirms this assertion and is proven in \cite{Durand07}.

\begin{theorem}{\rm\cite{Durand07}}\label{unique}
If $\h x_1$ and $\h x_2$ are two minimizers of $F_y$ (i.e. $\h x_1\in G_y$ and $\h x_2\in G_y$), then
\[ \nabla\Wi\h x_1  \propto \nabla\Wi\h x_2,~~~~~~~~\mbox{a.e. on }\Omega. \]
In other words, $\Wi \h x_1$ and $\Wi \h x_2$ have the same level lines.
\end{theorem}

In words, images $\Wi \h x_1$ and $\Wi \h x_2$ are obtained one from another by a local change of contrast
which is usually invisible for to the naked eye.

Some orientations for the choice of $\lambda_i$ were investigated in \cite{Durand07}.
If $i\in I_1$, the parameter $\lambda_i$ should be close to, but below the upper bound
$\|\wi_i\|_{\mbox{\footnotesize TV}}$, since above this bound, the coefficients $y[i]$ cannot be changed. For $i\in I_0$, a
reasonable choice is
\[ \lambda_i =  \max_{k\neq i}\left|\int_{\Omega}\left(\nabla \wi_i\right)^T
\frac{\nabla\wi_k}{\left|\nabla \wi_k\right|}\,ds\right| ~,
\]
where $.^T$ stands for transposed. If $\lambda_i$ is below this bound, some neighboring outliers might not be properly removed although Gibbs oscillations are better reduced. Another important remark is that, for some multiscale
transforms, the bounds discussed above are constant. In the proposed model, we use
only two values for $\lambda_i$, depending only on the set
$I_\epsilon$ the index $i$ belongs to.

We focus on the coefficients of a curvelets transforms of the log-data because (a) such a transform
captures efficiently the main features of the data and (b) it is a tight-frame which is helpful for the
subsequent numerical stage.

\section{Minimization for the log-image}\label{minsch}
\label{MinimAlgo}
Let us rewrite the minimization problem defined in \eqref{FoF} and \eqref{Psio}-\eqref{Phio} in a more compact form:
find $\h x$ such that $F_y(\h x)=\min_x F_y$ for
\beq
\label{minFoF}\barr{lll}
&&\min_{x} ~ F_y=\Psi+\Phi,\\~\\
\mbox{where}&&\left\{\barr{lll}
\Psi(x)&=&\disp{\norm{\Lambda(x-y_\HT)}_1,~~\mbox{for}~~\Lambda=\mathrm{diag}(\lambda_i)_{i \in I}},\\
\Phi(x)&=&\disp{\normtv{\Wi x}}.
\earr
\right.\earr\eeq
where $\lambda_i$ are the coefficients given in \eqref{Psio}.
Clearly, $\Psi$ and $\Phi$ are proper lower-semicontinuous convex functions, hence the same holds true for $F_y$.
The set $G_y$ introduced in \eqref{G} is non-empty by Theorem \ref{exist} and can be rewritten as
\[G_y=\{x \in \ell^2(I) \big| x \in (\partial F_y)^{-1}(0)\},\]
where $\partial F_y$ stands for subdifferential operator.
Minimizing $F_y$ amounts to solving the inclusion
\[
0 \in \partial F_y(x) ~,
\]
or equivalently, to finding a solution to the fixed point equation
\beqn\label{proxalgo}
x = (\Id + \gamma \partial F_y)^{-1}(x) ~,
\eeqn
where $(\Id + \gamma \partial F_y)^{-1}$ is the
{\textit{resolvent operator}} associated to $\partial F_y$, $\gamma>0$ is
the proximal stepsize  and $\Id$ is the identity map on the Hilbert space $\ell^2(I)$.
The proximal  schematic algorithm resulting from (\ref{proxalgo}), namely
 $$x^{(t+1)}=(\Id + \gamma \partial F_y)^{-1}(x^{(t)}),$$
  is a fundamental tool for finding the root of any maximal monotone operator
\cite{Rockafellar76,Eckstein92}, such as e.g. the subdifferential of a convex function.
Since the resolvent operator $(\Id + \gamma \partial F_y)^{-1}$ for $F_y$ in (\ref{minFoF}) cannot be calculated in
closed-form, we focus on iterative methods.

Splitting methods do not attempt to evaluate the resolvent mapping $(\Id + \gamma (\partial \Psi + \partial \Phi))^{-1}$
of the combined function $F_y$, but instead perform a sequence of calculations involving separately the
resolvent operators $(\Id + \gamma \partial \Psi)^{-1}$ and $(\Id + \gamma \partial \Phi))^{-1}$.
The latter are usually easier to evaluate, and this holds true for our functionals $\Psi$ and $\Phi$ in \eqref{minFoF}.

Splitting methods for monotone operators have numerous applications 
for convex optimization and monotone variational inequalities.
Even though the literature is abundant, these
can basically be systematized into three main classes:
the forward-backward \cite{Gabay83,Tseng91,Tseng00},
the Douglas/Peaceman-Rachford \cite{LionsMercier79}, and the little-used double-backward \cite{Lions78,Passty79}.
A recent theoretical overview of all these methods can be found in \cite{Combettes04a,Eckstein08}.
Forward-backward can be seen as a generalization of the classical gradient
projection method for constrained convex optimization, hence it
inherits all its restrictions. Typically, one must assume that either $\Psi$ or $\Phi$ is
differentiable with Lipschitz continuous gradient, and the stepsizes $\gamma$ must fall in a range
dictated by the gradient modulus of continuity; see \cite{CombettesWajs05} for an excellent account.
Obviously,  forward-backward splitting is not adapted to our functional (\ref{proxalgo}).

\subsection{Specialized Douglas-Rachford splitting algorithm}
The Douglas/Peaceman-Rachford family is the most general preexisting class of monotone operator splitting
methods. Given a fixed scalar $\gamma > 0$, let
\beq J_{\gamma\partial \Psi}\defeq(\Id + \gamma \partial \Psi)^{-1} ~~\text{ and }~~
J_{\gamma\partial \Phi}\defeq(\Id + \gamma \partial \Phi)^{-1}.\label{similar}
\eeq
Given a sequence $\mu_t \in (0,2)$, this class of methods can be expressed via the recursion
\beqn
\label{DRalgo}
x^{(t+1)} =   \crochets{\parenth{1-\frac{\mu_t}{2}}\Id + \frac{\mu_t}{2}(2J_{\gamma\partial \Psi}-\Id)\circ(2J_{\gamma\partial \Phi}-\Id)}x^{(t)} ~.
\eeqn

Since our problem \eqref{minFoF} admits solutions, the following result ensures that iteration \eqref{DRalgo} converges for our functional~$F_y$.
\begin{theorem}
\label{convergence}
Let $\gamma > 0$ and $\mu_t \in (0,2)$ be such that $\sum_{t\in\NN}\mu_t(2-\mu_t)=+\infty$.
Take $x^{(0)} \in \ell^2(I)$ and consider the sequence of iterates defined by (\ref{DRalgo}).
Then, $(x^{(t)})_{t \in \NN}$ converges weakly to some point $\h{x}\in\ell(I)$ and $J_{\gamma\partial \Phi}(\h{x}) \in G_y$.
\end{theorem}
This statement is a straightforward consequence of \cite[Corollary~5.2]{Combettes04a}.
For instance, the  sequence $\mu_t=1,\forall t\in\NN$, satisfies the requirement of the latter theorem.

It will be convenient to introduce the reflection operator
\beq\rprox_{\varphi} = 2\prox_{\varphi} - \Id.\label{reflex}\eeq
 where $\prox_\varphi$ is the proximity operator of $\varphi$ according to in Definition~\ref{def:1}.
 Using \eqref{sim} and \eqref{reflex}, the Douglas-Rachford iteration given in \eqref{DRalgo}
becomes
\beqn\label{DRproxalgo}
x^{(t+1)} =   \crochets{\parenth{1-\frac{\mu_t}{2}}\Id + \frac{\mu_t}{2}\rprox_{\gamma \Psi}\circ\rprox_{\gamma \Phi}}x^{(t)} ~.
\eeqn

Below we compute the resolvent operators
$J_{\gamma\partial \Psi}$  and $J_{\gamma\partial \Phi}$ with the help of Moreau proximity operators.

\subsection{Proximal calculus}
Proximity operators were inaugurated in \cite{Moreau1962} as a generalization of convex projection operators.
\begin{definition}[Moreau\cite{Moreau1962}]
\label{def:1}
Let $\varphi$ be a proper, lower-semicontinuous and convex function defined on a Hilbert space $\H$.
Then, for every $x\in\H$, the
function $z \mapsto \varphi(z) + \norm{x-z}^{2}/2$, for $z\in\H$,  achieves its infimum at a unique
point denoted by $\prox_{\varphi}x$. The operator $\prox_{\varphi} : \H \to \H$ thus
defined is the \textit{proximity operator} of $\varphi$.
\end{definition}
By the minimality condition for $\prox_{\varphi}$, it is straightforward that
$\forall x,p \in \H$ we have
\beq
\label{eq:prox}
p = \prox_{\varphi} x \iff x-p \in \partial\varphi(p) ~\iff~(\Id+\partial\varphi)^{-1}=\prox_{\varphi}.
\eeq
Then \eqref{similar} reads
\beq J_{\gamma\partial \Psi}=\prox_{\gamma\Psi} ~~\text{ and }~~
J_{\gamma\partial \Phi}=\prox_{\gamma\Phi}.
\label{sim}\eeq

\subsubsection{Proximity operator of $\Psi$}\label{proxPSI}
The proximity operator of $\gamma \Psi$ is {established in the lemma stated below.}
\begin{lemma}
\label{proxf1}
Let $x \in \ell^2(I)$. Then
\beqn
\label{eq:proxf1}
\prox_{\gamma \Psi}(x) = \parenth{y_\HT[i] + \ST^{\gamma\lambda_i}\parenth{x[i] - y_\HT[i]}}_{i \in I} ~,
\eeqn
with
\beq\ST^{\gamma\lambda_i}(z[i]) = \max\big\{0,\;z[i] - \gamma\lambda_i\sign(z[i])\big\}.\label{soft-threshold}\eeq
\end{lemma}

\begin{proof}
$\Psi$ is an additive separable function in each coordinate $i \in I$.
Thus, solving the proximal minimization problem of Definition \ref{def:1} is also separable.
For any convex function $\phi$ and $v \in \ell^2(I)$, put $\psi(.)=\phi(.-v)$.
Then
\begin{eqnarray*}
p = \prox_\psi(x) &\iff& x - p \in \partial\psi(p) \\
&\iff& (x - v) - (p - v) \in \partial\phi(p-v) \\
&\iff& p - v = \prox_\phi(x - v) \\
&\iff& p = v + \prox_\phi(x - v) ~.
\end{eqnarray*}
For each $i\in I$, we apply this result with $v=y_\HT[i]$ and $\phi(z[i])=\gamma\lambda_i |z[i]|$.
Noticing that $\prox_\phi=\ST^{\gamma\lambda_i}$ is soft-thresholding with threshold $\gamma\lambda_i$,
leads to~\eqref{eq:proxf1}.
\end{proof}

\bigskip

 Note that now
\beq\rprox_{\gamma \Psi}(x) = 2\parenth{y_\HT[i] + \ST^{\gamma\lambda_i}\parenth{x[i] - y_\HT[i]}}_{i \in I}-x ~.\label{rproxPsi}
\eeq

\subsubsection{Proximity operator of $\Phi$}
Clearly, $\Phi(x) = \normtv{\cdot}\circ\Wi(x)$ is a pre-composition of the TV-norm with the linear operator $\Wi$.
Computing the proximity operator of $\Phi$ for an arbitrary $\Wi$ may be intractable.
We adopt the following assumptions:
\bit
\item[(w1)] $\Wi:\ell^2(I)\to L^2(\O)$ is surjective;
\item[(w2)] $\Wi\W = \Id$ and $\Wi = c^{-1}\W^*$ for $0<c<\infty$, where $\W^*$ stands for  the adjoint operator;
note that we also have $W^*W=c~\Id$;
\item[(w3)] $\Wi$ is bounded.
\eit
 For any $ z(t)=(z_1(t),z_2(t))\in\RR^2, t \in \O$, we set $\left|z(t)\right|=\sqrt{z_1(t)^2 + z_2(t)^2}$.
Let $\X=L^2(\O) \times L^2(\O)$ and $\pds{\cdot}{\cdot}_\X$ be the inner product in $\X$, and $\nnorm{\cdot}_p$,
for $p\in[1,\infty]$ the $L^p$-norm on $\X$.
We define $\overline{B}^{\;\gamma}_\infty(\X)$ as the {\sf closed} $L^\infty$-ball of radius
$\gamma$ in $\X$,
\beq\overline{B}^{\;\gamma}_\infty\defeq\left\{z\in\X:\nnorm{z}_\infty\leq\gamma\right\}=
\Big\{z=(z_1,z_2)\in\X: |z(t)|\leq\gamma, \forall t \in \O\Big\},\label{Binfty}\eeq
and
$P_{\overline{B}^{\;\gamma}_\infty(\X)}: \X \to \overline{B}^{\;\gamma}_\infty(\X)$ the associated projector;
it is easy to check that the latter is equal to
the proximity operator of the indicator function of $\overline{B}^{\;\gamma}_\infty(\X)$.
The expression for $\prox_{\gamma\Phi}$  is given in the next lemma while the computation scheme to
solve item (ii) is stated in Lemma~\ref{lem:3}.
\begin{lemma}
\label{proxf2}
Let $x \in \ell^2(I)$ and $\overline{B}^{\;\gamma}_\infty(\X)$ is as defined above.
\bit
\item[(i)] Denoting by $\prox_{c^{-1}\gamma\normtv{\cdot}}(u)$ the proximity operator of the ($c^{-1}$-scaled) TV-norm, we have
\beqn
\label{eq:proxf2}
\prox_{\gamma \Phi}(x) = \parenth{\Id - \W \circ \parenth{\Id - \prox_{c^{-1}\gamma\normtv{\cdot}}} \circ \Wi}(x) ~;
\eeqn
\item[(ii)] Furthermore,
\beq
\prox_{c^{-1}\gamma\normtv{\cdot}}(u) = u - P_C(u) ~,
\label{prox:scaled}
\eeq
where
\beq C=\Big\{\mathrm{div}(z)\in L^2(\O) \big|z\in \C^\infty_c(\O\times\O), z \in
\overline{B}^{\;\gamma/c}_\infty(\X)\Big\}. \label{C}\eeq
\eit
\end{lemma}
\begin{proof}
Since $\Wi$ is surjective, its range is $L^2(\O)$ which is closed.
Moreover, the domain $\mathrm{dom}(\normtv{\cdot}) =L^2(\O)$  as well,  so
that $\mathrm{cone}\left(\mathrm{dom}\normtv{\cdot}-\mathrm{range}~\Wi\right)=\{0\}$  which is a closed subspace of $L^2(\O)$.
Reminding that $\normtv{\cdot}$ is lower bounded, continuous and convex, it is clear that all assumptions
required in \cite[Proposition 11]{Combettes07} are satisfied. Applying the same proposition yields statement~(i).

 We focus next on (ii).
Note that $C$ in \eqref{C} is a closed convex subset since $\overline{B}^{\;\gamma/c}_\infty(\X)$
is closed and convex, and
$\mathrm{div}$ is linear; thus the projection $P_C$ is well defined.

Let us remind that the Legendre-Fenchel (known also as the convex-conjugate) transform of a function
$\ph:\H\to\RR$, where $\H$ is an Hilbert space, is defined by
\[\varphi^*(w) = \sup_{u \in \mathrm{dom}(\varphi)} \big\{\pds{w}{u} - \varphi(u)\big\},\]
and that $\ph^*$ is a closed convex function.
If $\ph$ is convex, proper and lower semi-continuous,
the original Moreau decomposition \cite[Proposition 4.a]{Moreau1962} tells us that
\beq
\label{eq:proxconj}
\prox_{\varphi}+\prox_{\varphi^*} = \Id ~.
\eeq
One can see also \cite[Lemma~2.10]{CombettesWajs05} for an alternate proof.
It is easy to check that the conjugate function of a norm is the indicator function $\imath$ of the ball of its dual
norm, see e.g. \cite[Eq.(2.7)]{AA05}; thus
\[\big(c^{-1}\gamma\normtv{\cdot}\big)^*(z)=
\begin{cases}
0   &\mbox{if} ~ z\in C~ ,\\
+\infty &\mbox{if} ~ z\not\in C~,
\end{cases}
\]
where $C$ is given in \eqref{C}.
Using Definition \ref{def:1}, it is straightforward that
\[
\prox_{\big(c^{-1}\gamma\normtv{.}\big)^*}=P_C.
\]
Identifying $c^{-1}\gamma\normtv{.}$ with $\ph$ and $\big(c^{-1}\gamma\normtv{.}\big)^*$ with $\ph^*$,
equation \eqref{eq:proxconj} leads to statement~(ii). The proof is complete.
\end{proof}

\bigskip

Note that our argument \eqref{eq:proxconj} for the computation of $\prox_{c^{-1}\gamma\normtv{\cdot}}(u)$ is
not used  in \cite{Chambolle04}, which instead uses conjugates and bi-conjugates of the objective function.

\begin{remark}\rm
In view of equations \eqref{prox:scaled} and \eqref{C}, we one can see that
the term between the middle parentheses in equation \eqref{eq:proxf2} admits a simpler form:
\[\Id - \prox_{c^{-1}\gamma\normtv{\cdot}}=P_C.\]
\end{remark}

 Using \eqref{reflex} along with \eqref{eq:proxf2}-\eqref{prox:scaled} we easily find that
\beqn\rprox_{\gamma \Phi}(x) &=&
\parenth{\Id - 2\W \circ \parenth{\Id - \prox_{c^{-1}\gamma\normtv{\cdot}}} \circ \Wi}(x) ~ \label{rPHI} \nonumber\\
&=& \parenth{\Id - 2\W \circ P_C \circ \Wi}(x) ~.
\eeqn

\subsubsection{Calculation of the projection $\bm{P_C}$ in \eqref{prox:scaled} in a discrete setting}\label{PCD}

{\sl In what follows, we work in the discrete setting. We consider that
that $W$ is an $M\times N$ tight frame}
with $M=\#I\gg N$, admitting a  constant $c>0$ such that
$$\Wi\W = \Id~~~\mbox{and} ~~~\Wi = c^{-1}\W^T ~~(~\mbox{hence}~~~W^TW=c~\Id).$$
(This is the discrete equivalent of assumption (w2).)
 We also suppose that $\Wi:\ell^2(I)\to\ell^2(\O)$ is surjective.
Next we replace $\X$ by its discrete counterpart,
\beq\X=\ell^2(\O) \times \ell^2(\O)~~\mbox{where}~~\O~~\mbox{is discrete with}~~\#\O=N.
\label{Xdisc}\eeq
We denote the discrete gradient by $\nablaD$ and
consider $\divD: \X \to \ell^2(\O)$ the discrete divergence defined by analogy with the continuous
setting \footnotemark~as
the adjoint of the gradient $\divD=-{\nablaD}^*$; see \cite{Chambolle04}.
\footnotetext{More precisely, let $u\in\ell^2(\O)$ be of size $m\times n$, $N=mn$.
We write $(\nablaD u)[i,j]=\big(u[i+1,j]-u[i,j],~u[i,j+1]-u[i,j]\big)$ with boundary conditions
$u[m+1,i]=u[m,i],~\forall i$ and $u[i,n+1]=u[i,n],~\forall i$; then
for $z\in\X$, we have $(\divD(z))[i,j]=\big(z_1[i,j]-z_1[i-1,j]\big)+\big(z_2[i,j]-z_2[i,j-1]\big)$
along with $z_1[0,i]=z_1[m,i]=z_2[i,0]=z_2[i,n]=0$, $\forall i$.}

Unfortunately, the projection in \eqref{prox:scaled} does not admit an explicit form.
The next lemma provides an iterative scheme to compute the proximal points introduced in Lemma~\ref{proxf2}.
In this discrete setting, $C$ in \eqref{C} admits a simpler expression:
\beq C=\Big\{\divD(z)\in \ell^2(\O) ~\big|~ z \in
\overline{B}^{\;\gamma/c}_\infty(\X)\Big\}.
 ~,
\label{C1}
\eeq
 where $\overline{B}^{\;\gamma/c}_\infty(\X)$ is defined according to \eqref{Binfty}.
\begin{lemma}\label{lem:3}
\label{proxf3}
We adapt all assumptions of Lemma~\ref{proxf2} to the new discrete setting, as explained above.
Consider the forward-backward iteration
\beqn
\label{eq:proxtv}
z^{(t+1)} = P_{\overline{B}^{\;1}_\infty(\X)}\parenth{z^{(t)} + \beta_t \nablaD\parenth{\divD(z^{(t)}) - c u/\gamma}} ~ \mbox{for}
~~~0 < \inf_t \beta_t \leq \sup_t \beta_t < 1/4,
\eeqn
where $\forall (i,j) \in \O$
\[\disp{P_{\overline{B}^{\;1}_\infty(\X)}(z)[i,j] =
\begin{cases}
z[i,j]            & \mbox{if} ~ |z[i,j]| \leq 1;\\
\disp{\frac{z[i,j]}{|z[i,j]|}}     & \mbox{otherwise} ~. 
\end{cases}} ~
\]

Then
\ben
\item[(i)]
$(z^{(t)})_{t \in \NN}$ converges to a point $\h{z} \in \overline{B}^{\;1}_\infty(\X)$;
\item[(ii)]
$\disp{\parenth{c^{-1}\gamma\divD(z^{(t)})}_{t \in \NN}}$ converges to
$\disp{c^{-1}\gamma\divD(\h{z})=(\Id - \prox_{c^{-1}\gamma\normtv{\cdot}})(u)}$.
\een
\end{lemma}

\begin{proof}
Given $u\in\ell^2(\O)$,
the projection $\hat w=P_C(u)$ defined by \eqref{prox:scaled} and \eqref{C1} is unique and satisfies
\beqn
\label{eq:projC} \h w&=&\arg\min_{w\in C}\frac{1}{2}\|u-w\|^2=
\arg\min\left\{\frac{1}{2}\norm{\frac{c}{\gamma}u - w}^2~\mbox{subject to}~
w=\divD(z) ~\mbox{for}~z \in \overline{B}^{\;1}_\infty(\X)\right\}\nonumber\\
&&\Updownarrow\nonumber\\
\h w&=&\divD(\h z)~~\mbox{where}~~\h z=
\arg\min_{z \in \overline{B}^{\;1}_\infty(\X)} ~ \frac{1}{2}\norm{\frac{c}{\gamma}u - \divD(z)}^2
\eeqn
This problem can be solved using a projected gradient method (which is a special instance of the forward-backward
splitting scheme) whose iteration is given by \eqref{eq:proxtv}.
This iteration converges weakly to a minimizer of (\ref{eq:projC})---see
\cite[Corollary~6.5]{Combettes04a},
provided that the stepsize $\beta_t > 0$ satisfies $\sup_t \beta_t < 2/\delta^2$, where $\delta$ is the spectral norm of
the $\mathrm{div}$ operator. It is easy to check that $\delta^2 \leq 8$---see e.g. \cite{Chambolle04}.

Set
$$\omega^{(t)}=c^{-1}\gamma\divD(z^{(t)}), \forall t \in \NN~~\mbox{and}~~
\h{\omega}=c^{-1}\gamma\divD(\h{z}).$$ Thus,
\beqn
\label{eq:strong1}
\norm{\omega^{(t)} - \h{\omega}}^2 &=&  \parenth{\frac{\gamma}{c}}^2\norm{\divD(z^{(t)})-\divD(\h{z})}^2 \nonumber \\
                   &=& \parenth{\frac{\gamma}{c}}^2\pds{-\nablaD\parenth{\divD(z^{(t)}) -
\divD(\h{z})}}{z^{(t)}-\h{z}}_\X ~,
\eeqn
where we use the fact that $-\nablaD$ is the adjoint of $\divD$. Let $\gradf_z$ denote the gradient of a scalar-valued function of
 $z$, not to be confused with the discrete gradient operator $\nablaD$ of an image.
 The gradient of the function $\frac{1}{2}\norm{cu/\gamma - \divD(z)}^2$ with respect to
 $z$ is $\gradf_z\parenth{{\frac{1}{2}\norm{cu/\gamma -\divD(z)}^2}}=-\nablaD\parenth{\divD(z) - c u/\gamma}$.
This relation together with the Schwarz inequality applied to \eqref{eq:strong1} lead to
\beqn
\label{eq:strong2}
\norm{\omega^{(t)} - \h{\omega}}^2 &\leq& \parenth{\frac{\gamma}{c}}^2\nnorm{\nablaD~\divD(z^{(t)}) -
\nablaD~\divD(\h{z})}_2\nnorm{z^{(t)}-\h{z}}_2 \nonumber \\
&=&\parenth{\frac{\gamma}{c}}^2\nnorm{\nablaD~\big(\divD(z^{(t)}) - c u/\gamma \big)-
\nablaD~\big(\divD\big((\h{z})- c u/\gamma \big)}_2\nnorm{z^{(t)}-\h{z}}_2 \nonumber\\
&=&  {0.5}\parenth{\frac{\gamma}{c}}^2\nnorm{\gradf_z\parenth{\norm{cu/\gamma-\divD(z^{(t)})}^2}  -
\gradf_z\parenth{\norm{cu/\gamma-\divD(\h{z})}^2}}_2\nnorm{z^{(t)}-\h{z}}_2 ~.
\eeqn
 From \cite[Theorem 6.3]{Combettes04a}, we deduce that the series
\[\sum_{t \in \NN} \nnorm{\gradf_z\parenth{\norm{cu/\gamma -\divD(\cdot)}}(z^{(t)}) -
\gradf_z\parenth{\norm{cu/\gamma -\divD(.)}}(\h{z})}_2^2\]
is convergent.
Inserting this property in (\ref{eq:strong2}) and using  the fact that the sequence
$(z^{(t)})_{t \in \NN}$ is bounded (as it converges weakly with
$\nnorm{\h{z}}_2 < {\lim\inf}_{t \to \infty}\nnorm{z^{(t)}}_2$), it follows that $\omega^{(t)}$ converges
strongly to $\h{\omega}$.
This completes the proof.
\end{proof}

\bigskip

The forward-backward splitting-based iteration proposed in \eqref{eq:proxtv} to
compute the proximity operator of the TV-norm is new and different from the projection
algorithm of Chambolle \cite{Chambolle04},
even tough the two algorithms bear some similarities.
The forward-backward splitting allows to derive a sharper upper-bound on the stepsize
$\beta_t$ than the one proposed in \cite{Chambolle04}{\textemdash}actually twice as large.
Let us remind that it was observed in \cite{Chambolle04} that the bound $1/4$ still works in practice.
Here we prove why thus is really true.

\subsection{Comments on the Douglas-Rachford scheme for $F_y$}
A crucial property of the Douglas-Rachford splitting scheme \eqref{DRproxalgo}
is its robustness to numerical errors that may occur when computing the proximity operators
$\prox_{\Psi}$ and $\prox_{\Phi}$, see \cite{Combettes04a}.
We have deliberately omitted this property in \eqref{DRproxalgo} for the sake of simplicity.
This robustness property has important consequences: e.g. it allows us to run
the forward-backward sub-recursion \eqref{eq:proxtv} only a few iterations to compute
an approximate of the TV-norm proximity operator in the inner iterations,
and the Douglas-Rachford is still guaranteed to converge provided that these numerical errors are under control.
More precisely, let $a_t \in \ell^2(I)$ be an error term that models the inexact computation of $\prox_{\gamma\Phi}$
in \eqref{eq:proxf2}, as the latter is obtained through \eqref{eq:proxtv}.
If the sequence of error terms $\parenth{a_t}_{t \in \NN}$ and stepsizes $\parenth{\mu_t}_{t \in \NN}$
defined in Theorem~\ref{convergence} obey $\sum_{t\in\NN} \mu_t\norm{a_t} < +\infty$,
then the Douglas-Rachford algorithm \eqref{DRproxalgo} converges weakly \cite[Corollary~6.2]{Combettes04a}.
In our case, using 200 inner iterations in \eqref{eq:proxtv} was sufficient to satisfy this requirement.

\begin{remark} \rm
Owing to the splitting framework and proximal calculus, we have shown in Lemma~\ref{proxf2}
 that the bottleneck of the minimization algorithm is in the computation of the proximity-operator
 of the TV-norm. In fact, computing $\prox_{\normtv{\cdot}}$ amounts to solving a discrete ROF-denoising.
 Our forward-backward iteration is one possibility among others, and other algorithms
 beside \cite{Chambolle04} have been proposed to solve the discrete ROF-denoising  problem.
 While this paper was submitted, our attention was drawn to an independent work of \cite{Aujol08} who,
 using a different framework, derive an iteration similar to (\ref{eq:proxtv})
 to solve the ROF. Another parallel work of \cite{Zhu08} propose an application of gradient projection
to solving the dual problem \eqref{eq:projC}.
We are of course aware of max-flow/min-cut type algorithms, for instance the one in
 \cite{ChambolleDarbon08}. We have compared our whole denoising procedure using our implementation
 of $\prox_{\normtv{\cdot}}$ and the max-flow based implementation that we adapted from the code
 available at \cite{ChambolleTVURL}.
 We obtained similar results, although the max-flow-based algorithm was faster,
 mainly because it uses the $\ell^1$ approximation of the discrete gradient,
namely $\left\|(\nablaD u)[i,j]\right\|_1=\big|u[i+1,j]-u[i,j]\big|+\big|u[i,j+1]-[i,j]\big|$.
 Let us remind that this approximation for the discrete gradient does not inherits the rotational
 invariance property of the $L^2$ norm of the usual gradient.
\end{remark}

\section{Bias correction to recover the sought-after image}
\label{resto-expo}
Recall from \eqref{logN} that $u_0=\log S_0$ and set $\h u =\Wi \h x^{(N_{\mathrm{DR}})}$
as the estimator of $u_0$, where $N_{\mathrm{DR}}$ is the number of Douglas-Rachford iterations in \eqref{DRproxalgo}.
Unfortunately, the estimator $\h u$ is prone to bias, i.e. $\E{\h u}=u_0-b_{\h u}$.
A problem that classically arises in statistical estimation is how to correct such a bias.
More importantly is how this bias affects the estimate after applying the inverse transformation,
here the exponential.
Our goal is then to ensure that for the estimate $\h S$ of the image, we have
$\E{\h S}=S_0$.
Expanding $\h S$ in the neighborhood of $\E{\h u}$, we have
\beqn
\E{\exp{\h u}} &=& \exp{\parenth{\E{\h u }}} (1 + \var{\h u} /2 + R_2) \nonumber \\
                &=& S_0\exp{\parenth{-b_{\h u}}}(1 + \var{\h u} /2 + R_2) ~,
\eeqn
where $R_2$ is expectation of the Lagrange remainder in the Taylor series.
 One can observe that  the posterior distribution of $\h u$ is nearly symmetric, in which case $R_2 \approx 0$.
That is, $b_{\h u} \approx \log(1 + \var{\h u} /2)$ to ensure unbiasedness.
Consequently, finite sample (nearly) unbiased estimates of $u_0$ and $S_0$ are respectively
$\h u+\log(1 + \var{\h u} /2)$, and $\exp\parenth{\h u} (1+\var{\h u} /2)$.
$\var{\h u}$ can be reasonably estimated by $\psi_1(K)$, the variance of the noise $n$ in \eqref{logN} being given in
\eqref{var}. Thus, given the restored log-image $\h u$, our restored image read:
\beq\h S=\mathrm{exp}\parenth{\h u}(1+\psi_1(K)/2) ~.
\label{SB}\eeq

The authors of \cite{Xie02} propose a direct estimate of the bias $b_{\h u}$ using the obvious argument
that the noise $n$ in the log-transformed image has a non-zero mean $\psi_0(K)-\log K$.
A quick study shows that the functions $(1+\psi_1(K)/2)$ and $\exp(\log K-\psi_0(K))$ are very close for $K$ reasonably large.
Thus, the two bias corrections are equivalent.
Even though the bias correction approach we propose can be used in a more general setting.

\section{Full algorithm to suppress multiplicative noise} \label{DA}
Now, piecing together Lemma~\ref{proxf1}, Lemma~\ref{proxf2} and Theorem~\ref{convergence},
we arrive at the multiplicative noise removal algorithm:
{\def\baselinestretch{1}
\medskip\hrule\medskip
\noindent{\bf{Task:}} Denoise  an image $S$ contaminated with multiplicative noise according to \eqref{expo}. \\
\noindent{\bf{Parameters:}} The observed noisy image $S$, number of iterations $N_{\mathrm{DR}}$
(Douglas-Rachford outer iterations) and $N_{\mathrm{FB}}$ (Forward-Backward inner iterations),
stepsizes $\mu_t \in (0,2)$, $0<\beta_t < 1/4$ and $\gamma > 0$, tight-frame transform $W$ and initial threshold $T$
(e.g. $T=2\sqrt{\psi_1(K)}$), regularization parameters $\lambda_{0,1}$ associated to the sets $I_{0,1}$.\\
\noindent{\bf{Specific operators:}}
\bit
\item[(a)] $\ST^{\gamma\lambda_i}(z)=\Big(\max\big\{0,\;z[i] - \gamma\lambda_i\sign(z[i])\big\}\Big)_{i\in I},
    ~~~\forall z\in\RR^{\# I}$.
\item[(b)] $\disp{P_{\overline{B}^{\;1}_\infty(\X)}(z)[i,j] =
\begin{cases}
z[i,j]            & \mbox{if} ~ |z[i,j]| \leq 1;\\
\disp{\frac{z[i,j]}{|z[i,j]|}}     & \mbox{otherwise},
\end{cases}}$~~~~ $\forall (i,j) \in \O$.
\item[(c)] $\nablaD$ and $\divD$---the discrete versions of the continuous operators ~$\nabla$ and div.
\item[(d)] $\psi_1(\cdot)$ defined according to \eqref{polyG} (built-in Matlab function, otherwise see \cite{Press92}).
\eit
\noindent{\bf{Initialization:}}
\begin{itemize}
\item Compute $v=\log S$ and transform coefficients $y = \W v$. Hard-threshold $y$ at $T$ to get $y_\HT$.
Choose an initial $x^{(0)}$.
\end{itemize}
\noindent{\bf{Main iteration:}} \\
\noindent{\bf{For}} $t=1$ {\bf{to}} $N_{\mathrm{DR}}$,
\begin{itemize}
\item[(1)] Inverse curvelet transform of $x^{(t)}$ according to $u^{(t)}=\Wi x^{(t)}$.
\item[(2)]  Initialize $z^{(0)}$; {\bf{For}} $s=0$  {\bf{to}}  $N_{\mathrm{FB}}-1$
\bit
\item[~]
$z^{(s+1)} = P_{\overline{B}^{\;1}_\infty(\X)}\parenth{z^{(s)} + \beta_t \nablaD\parenth{\divD(z^{(s)}) - \frac{c}{\gamma} u^{(t)}}}$.
\eit
\item[(3)] Set $z^{(t)}=z^{(N_{\mathrm{FB})}}$.
\item[(4)] Compute $w^{(t)}=c^{-1}\gamma~\divD(z^{(t)})$.
\item[(5)] Forward curvelet transform: $\alpha^{(t)} = \W w^{(t)}$.
\item[(6)] Compute $r^{(t)}=\rprox_{\gamma \Phi}(x^{(t)}) =  x^{(t)} - 2\alpha^{(t)} $.
\item[(7)] By \eqref{rproxPsi} compute
$\disp{q^{(t)}=\left(\rprox_{\gamma \Psi}\circ\rprox_{\gamma \Phi}\right)x^{(t)}
=2\parenth{y_\HT[i] + \ST^{\gamma\lambda_i}\parenth{r^{(t)}[i] - y_\HT[i]}}_{i \in I}-r^{(t)} ~.}$
\item[(8)] Update $x^{(t+1)}$ using \eqref{DRproxalgo}: ~
$\disp{x^{(t+1)} =\parenth{1-\frac{\mu_t}{2}}x^{(t)}+\frac{\mu_t}{2}q^{(t)}}~.$
\eit
\noindent{\bf{End main iteration}} \\
\noindent{\bf{Output:}} Denoised image $\h{S}=\mathrm{exp}\parenth{\Wi x^{( N_{\mathrm{DR }})}}(1+\psi_1(K)/2)$.
\hrule\medskip}


\begin{remark}[Computation load]\rm
The bulk of computation of our denoising algorithm is invested in applying
$\W$ and its pseudo-inverse $\Wi$. These operators are of course never constructed explicitly,
rather they are implemented as fast implicit analysis and synthesis operators.
Each application of $\W$ or $\Wi$ cost $\mathcal{O}(N\log N)$ for the second generation
curvelet transform of an $N$-pixel image \cite{CandesFDCT05}.
If we define $N_{\mathrm{DR}}$ and $N_{\mathrm{FB}}$ as the number of iterations in the
 Douglas-Rachford algorithm and the forward-backward sub-iteration, the computational complexity
 of the denoising algorithm is of order $N_{\mathrm{DR}}N_{\mathrm{FB}}2N\log N$ operations.
\end{remark}

\section{Experiments}
\label{Exper}
In all experiments carried out in this paper, our algorithm was run
using second-generation curvelet tight frame along
with the following set of parameters:
$\forall t, \mu_t \equiv 1$, $\beta_t=0.24$, $\gamma = 10$ and $N_{\mathrm{DR}}=50$.
The initial threshold $T$ was set to $2\sqrt{\psi_1(K)}$.
For comparison purposes, some very recent multiplicative noise removal algorithms from the literature are considered:
the AA algorithm \cite{AubertAujol08} minimizing the criterion in \eqref{AA},
and the Stein-Block denoising method \cite{Chesneau08} in the curvelet domain,
applied on the log transformed image.
The latter is a sophisticated shrinkage-based denoiser that thresholds the coefficients by blocks rather than individually,
and has been shown to be nearly minimax over a large class of images in presence of additive bounded noise
(not necessarily Gaussian nor independent).
 We also tried the ``naive'' method, called L2-TV, where
$\h u$ minimizes \eqref{L2TV} and the denoised image is given after bias correcion according to \eqref{SB}.
No without surprise, one realizes that the results are quite good, even though some persistent outliers remain quite visible.
 This again raises the persistent question of relevance of PSNR (or even MAE) as a measure of perceptual restoration quality.
For fair comparison, the hyperparameters for all competitors were tweaked to reach their best level of performance on
each noisy realization.

The denoising algorithms were tested on three images: Shepp-Logan phantom, Lena
and Boat all of size $256 \times 256$ and with gray-scale in the range $[1,256]$.
For each image, a noisy observation is generated by multiplying the original image by a realization
of noise according to the model in \eqref{expo}-\eqref{Gamma} with the choice $\mu=1$ and $K=10$.
For a $N$-pixel noise-free image $S_0$ and its denoised version by any algorithm $\h{S}$,
the denoising performance is measured in terms of peak signal-to-noise ratio (PSNR) in decibels (dB)
\[
\mathrm{PSNR} = 20 \log_{10} \frac{\sqrt{N}\|S_0\|_\infty}{\norm{\h{S}-S_0}} ~ \mathrm{dB} ~,
\]
and mean absolute-deviation MAE
\[
\mathrm{MAE} = \norm{\h{S}-S_0}_1/N ~.
\]
The results are depicted in Fig.~\ref{fig:phantom}, Fig.~\ref{fig:lena} and Fig.~\ref{fig:boat}.
Our denoiser clearly outperforms its competitors both visually and quantitatively as revealed by the PSNR and MAE values.
The PSNR improvement brought by our approach is up to 4dB on the Shepp-Logan phantom, and is $\sim 1$dB for Lena and Boat.
Note also that a systematic behavior of AA algorithm is its tendency to lose some important details and the persistence
of a low-frequency ghost as it can be seen on the error maps on the third row in Figs.~\ref{fig:lena}~and~\ref{fig:boat}.

\setlength{\unitlength}{1mm}
\noindent\begin{figure}[ht]\vspace{-0.8cm}\centering
\begin{tabular}{cc}
\epsfig{figure=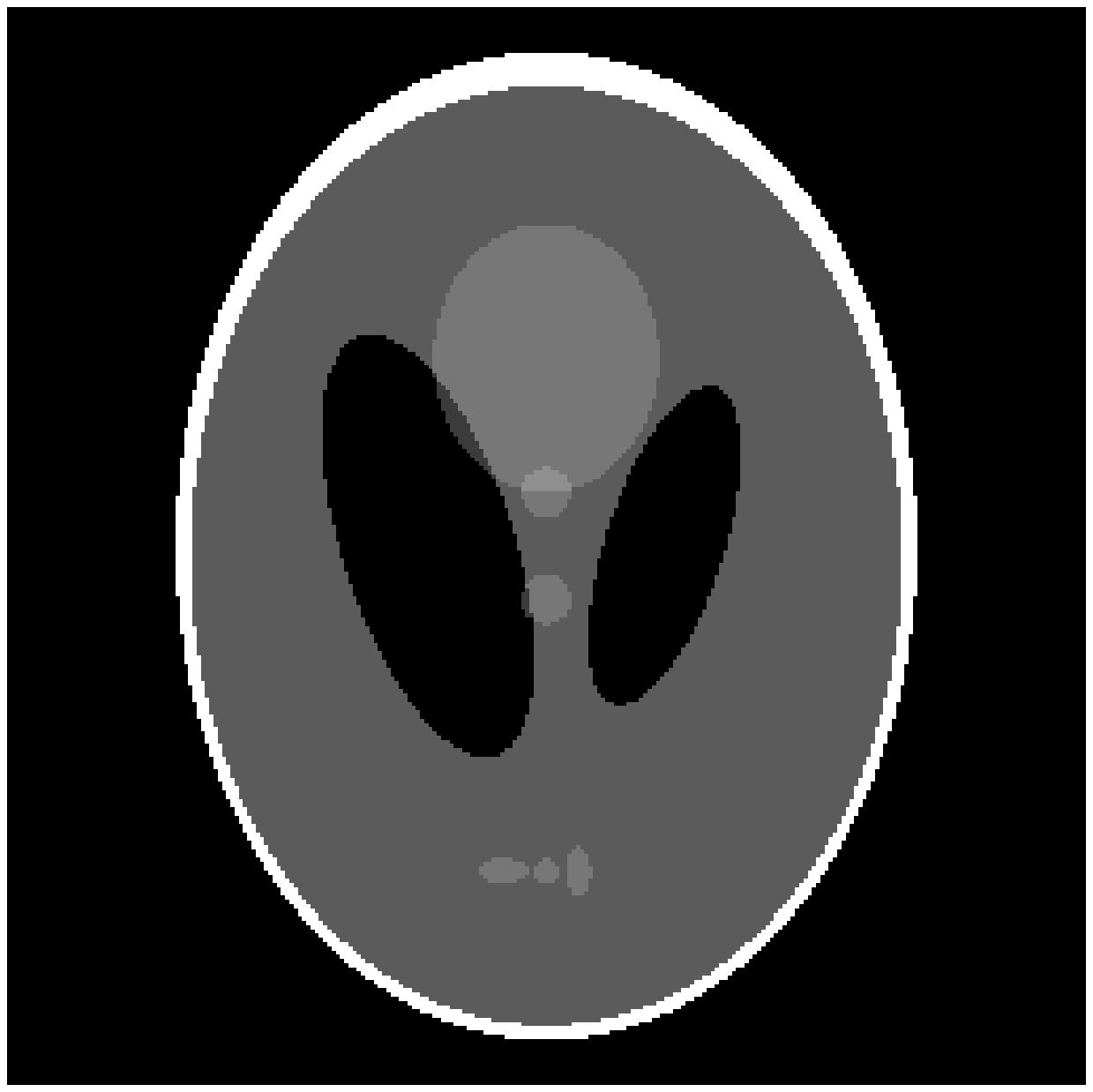,width=6.2cm,height=6.2cm}&
\epsfig{figure=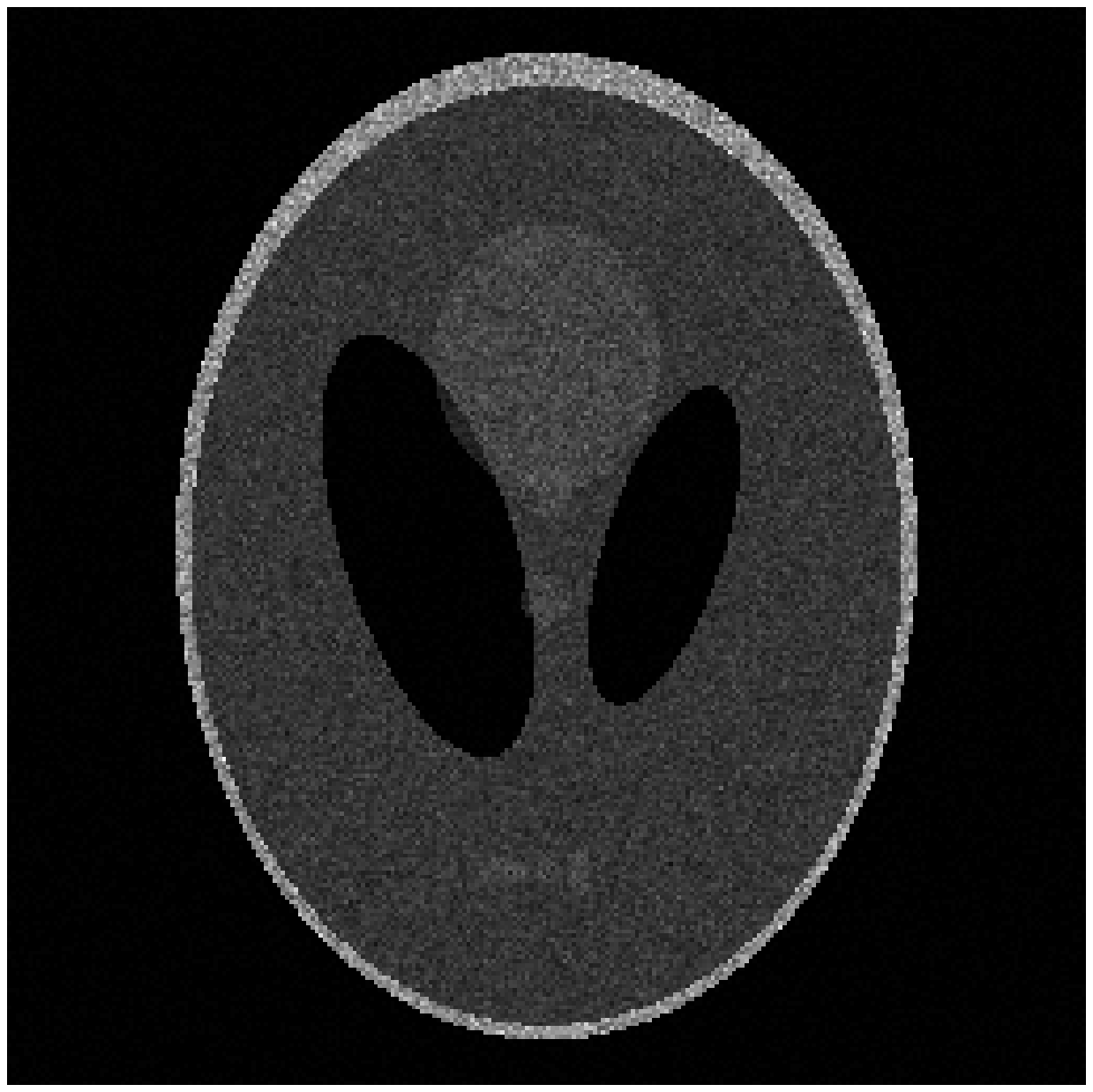,width=6.2cm,height=6.2cm}\\
{(a) Shepp-Logan ($256 \times 256$)}&{(b) Noisy $\mu=1, ~K=10$}\\
\epsfig{figure=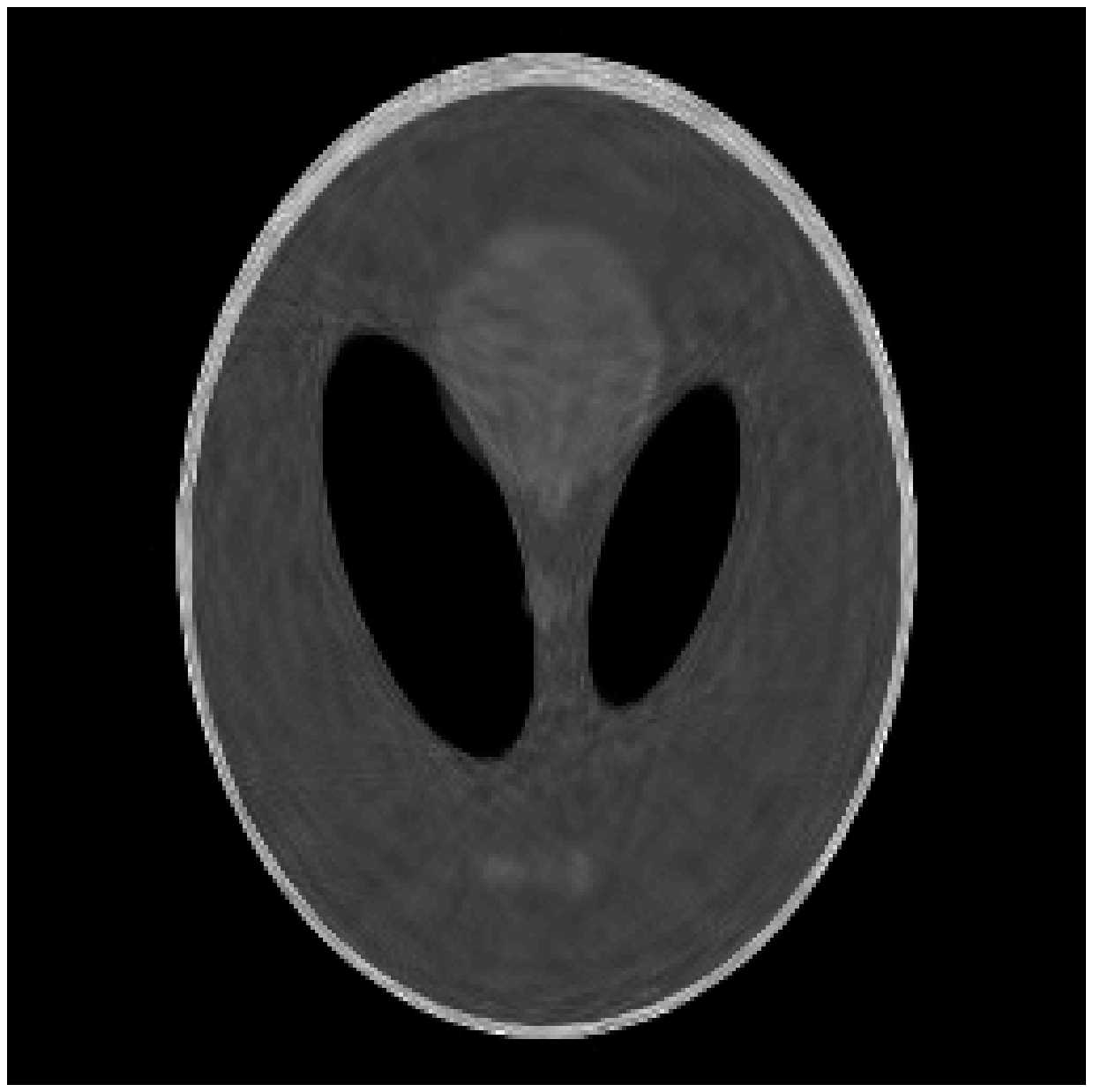,width=6.2cm,height=6.2cm}&
\epsfig{figure=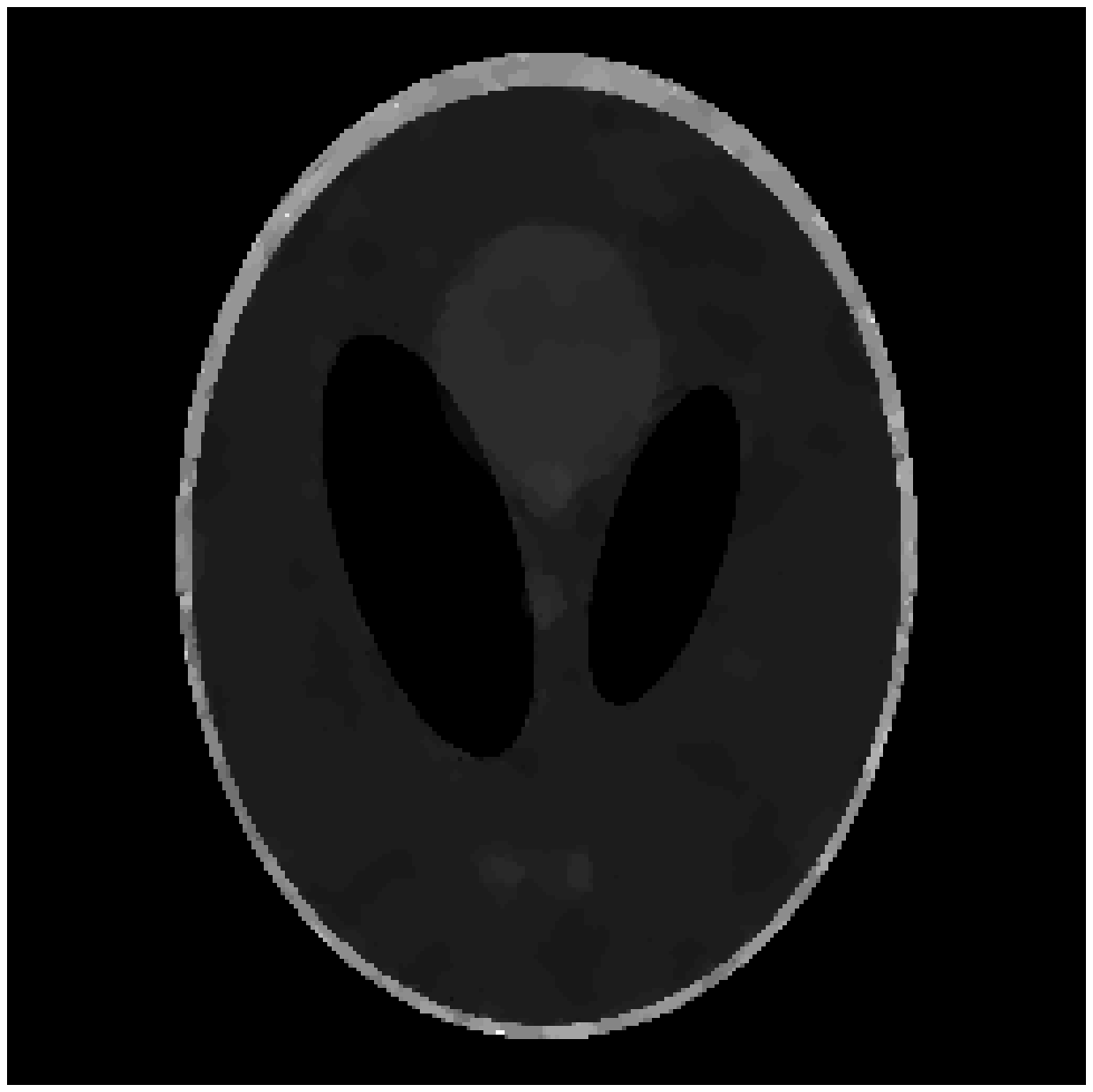,width=6.2cm,height=6.2cm}\\
{(c) Stein-block thresholding}&{(d) Our method}\\
\epsfig{figure=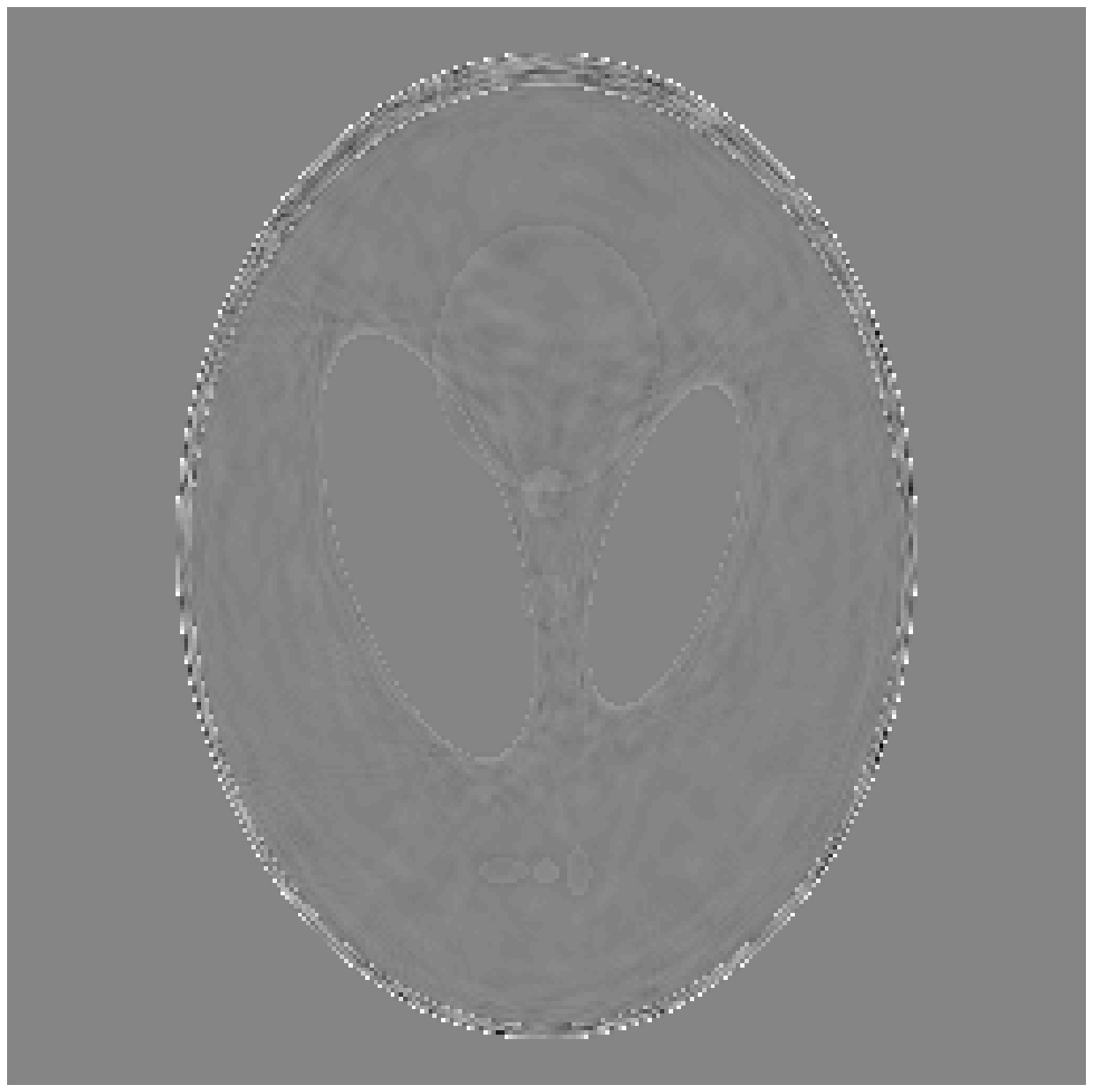,width=6.2cm,height=6.2cm}&
\epsfig{figure=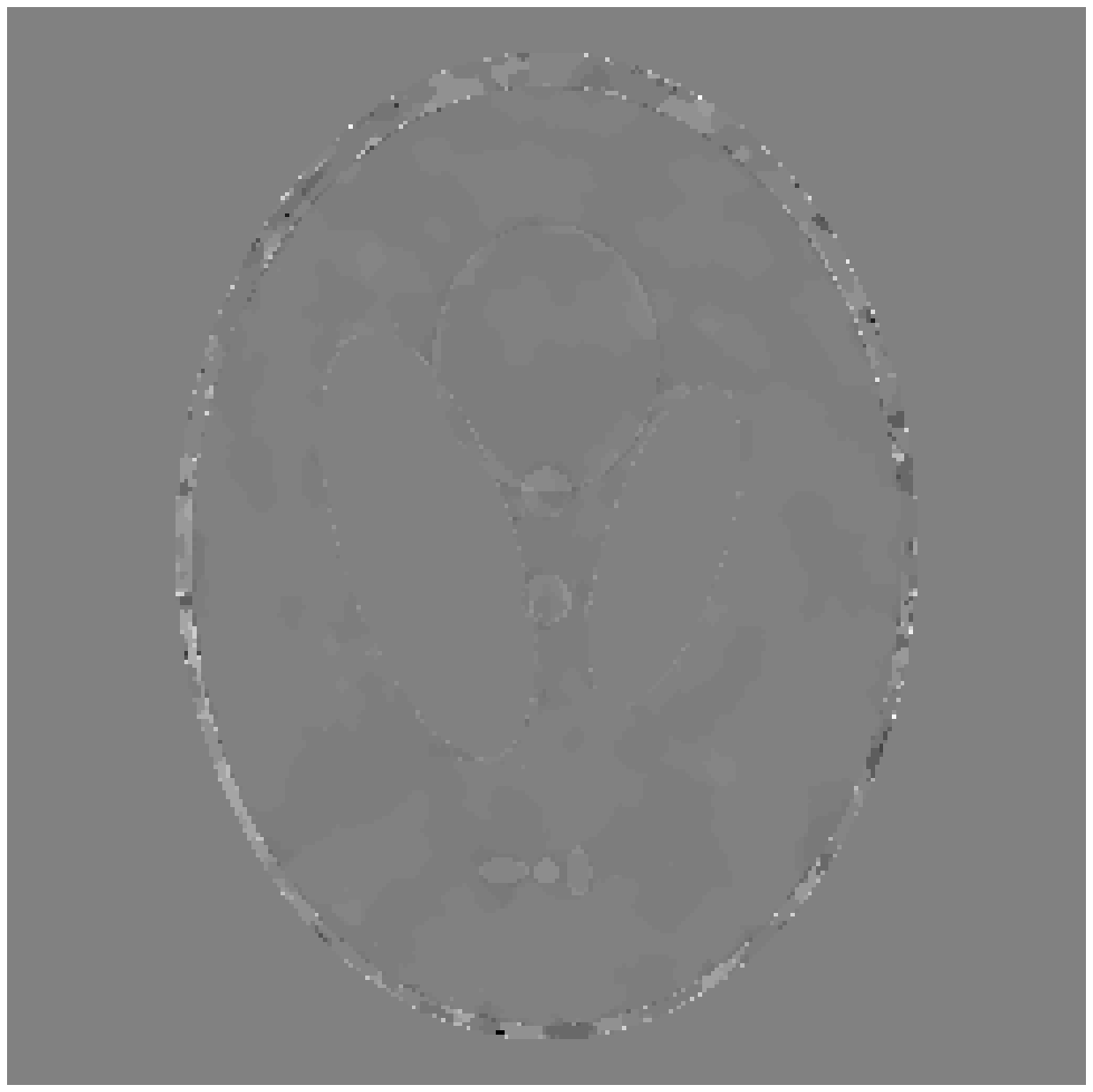,width=6.2cm,height=6.2cm}\\
{(e)}&{(f)}
\end{tabular}
\caption{Performance comparison with Shepp-Logan phantom ($256 \times 256$).
(a) Original. (b) Noisy $\mu=1, K=10$. (c) Denoised with Stein-block thresholding in
the curvelet domain \cite{Chesneau08} PSNR=24.73dB, MAE=4.
(d) Denoised with our algorithm PSNR=31.25dB, MAE=1.87. (e)-(f) Errors (restored $-$ original) for (c)-(d).}
\label{fig:phantom}
\end{figure}

\noindent\begin{figure}[ht]\vspace{-0.8cm}
\begin{tabular}{ccc}
\epsfig{figure=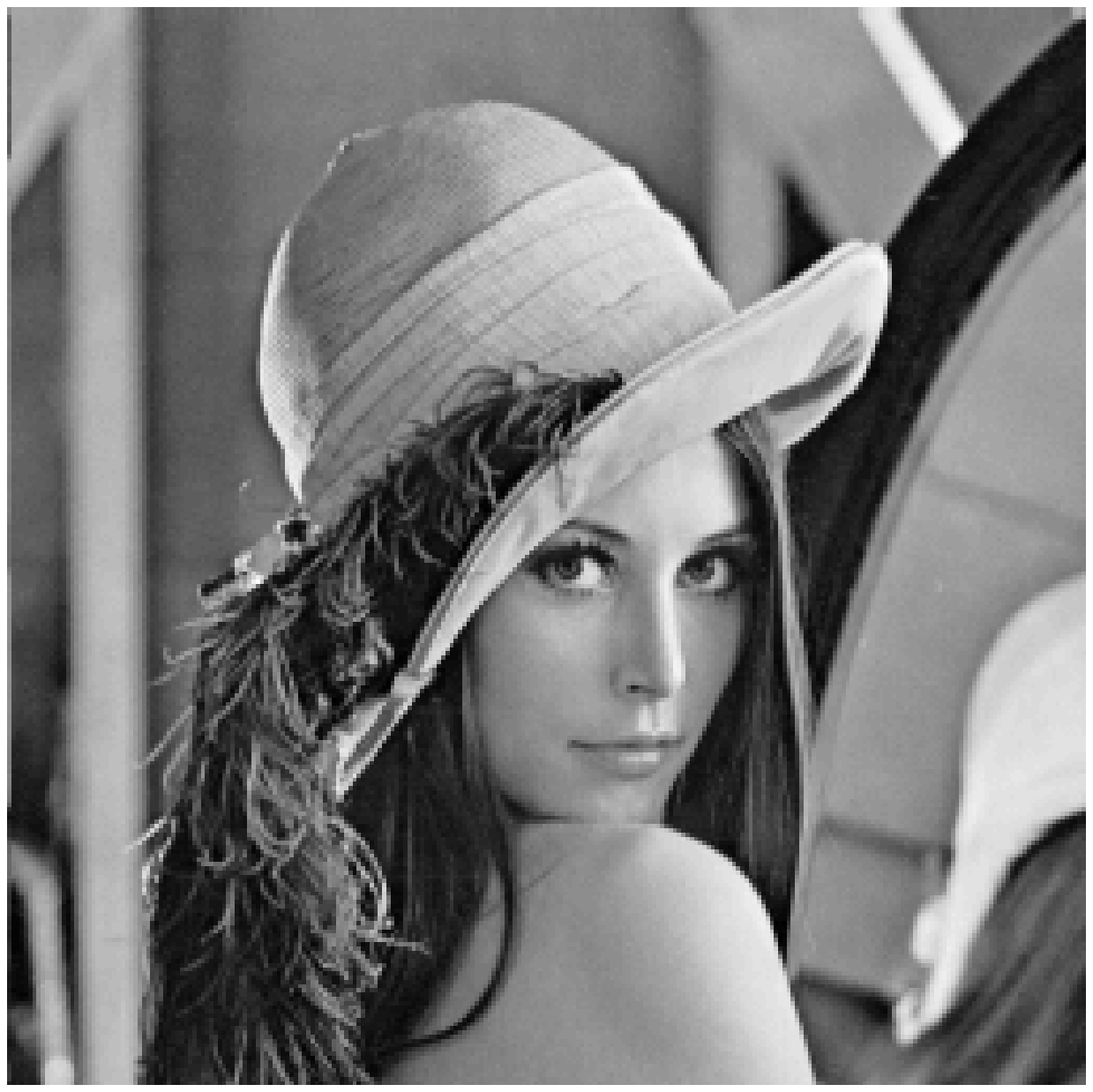,width=5cm,height=5cm}&
\epsfig{figure=lenanoisy.eps,width=5cm,height=5cm}&
\epsfig{figure=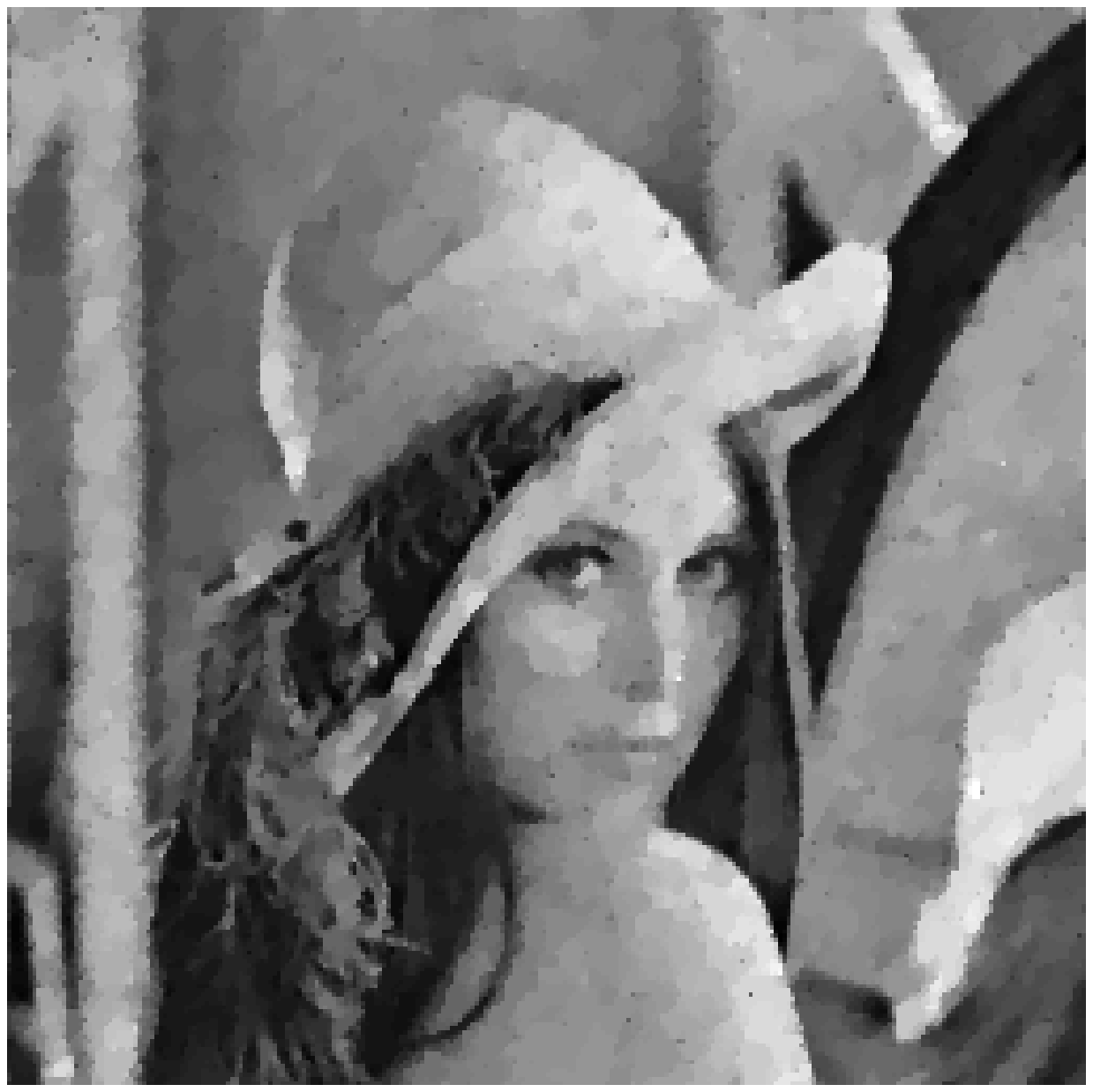,width=5cm,height=5cm}\\
{(a) Lena ($256\times 256$)---original}&{(b) Noisy: $\mu=1, ~K=10$}&{(c) L2-TV}\\
{}&{}&{{\sc psnr}=26.22  d\sc{b,~mae}=8.5  }\\
\epsfig{figure=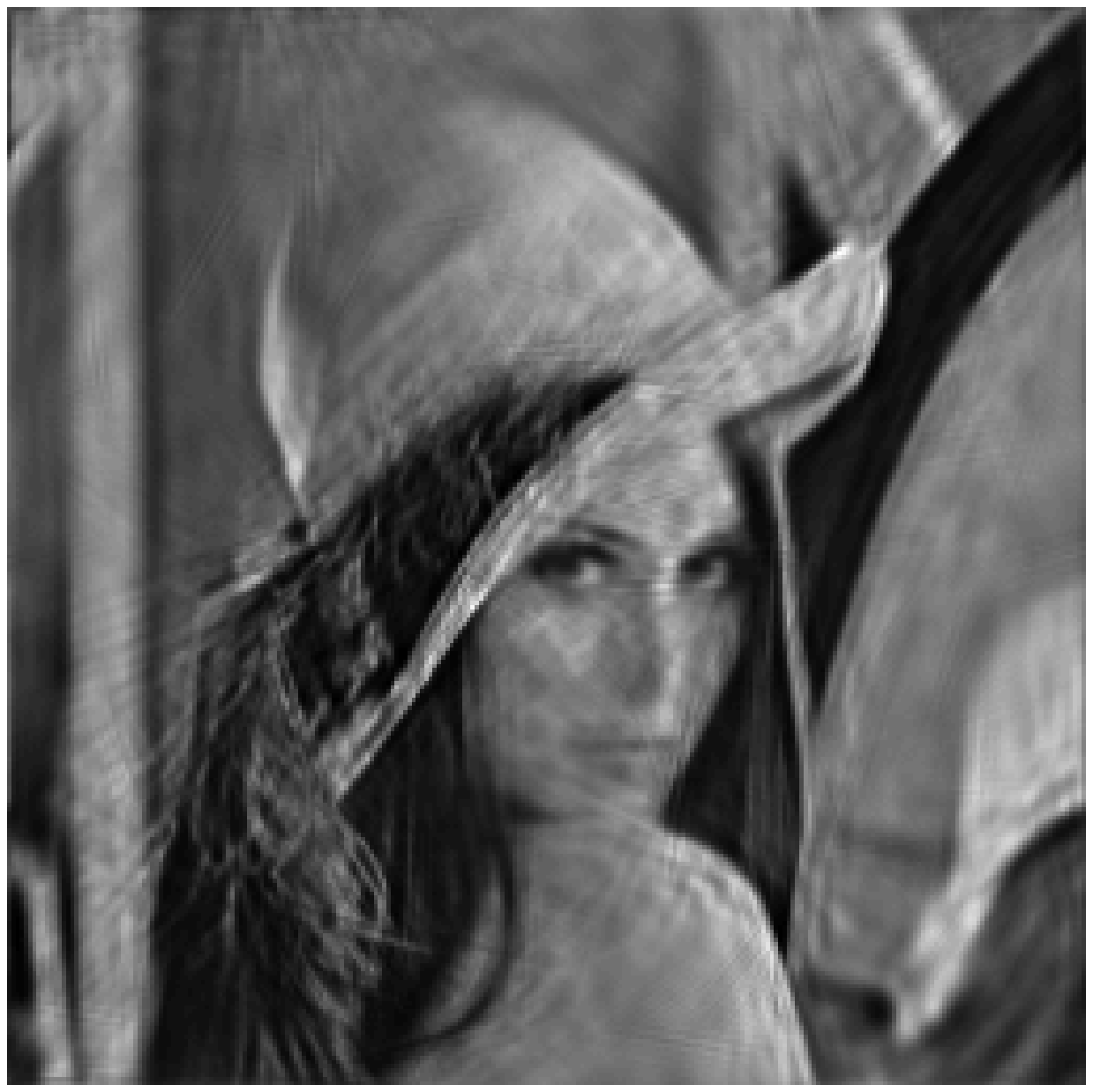,width=5cm,height=5cm}&
\epsfig{figure=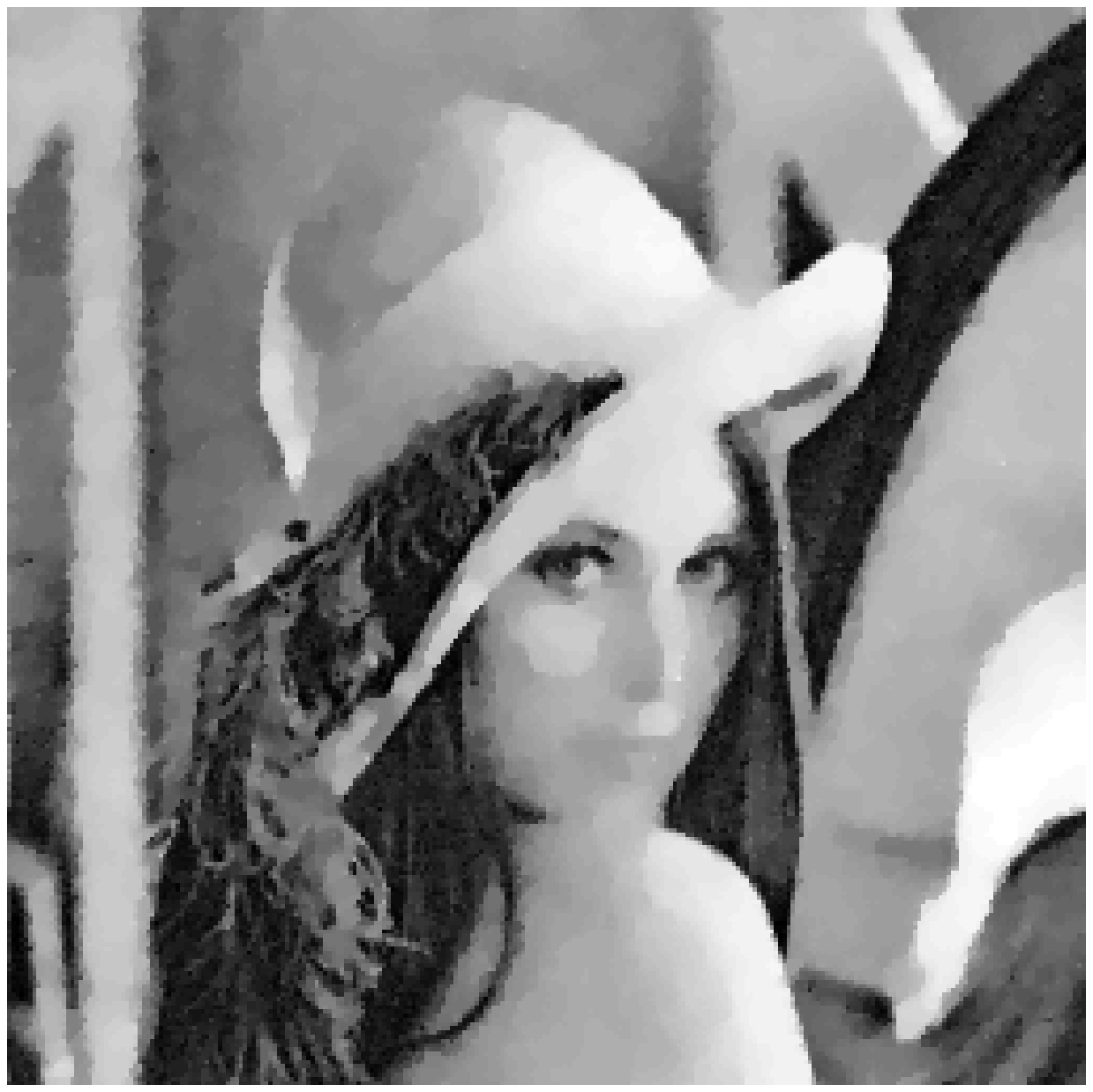,width=5cm,height=5cm}&
\epsfig{figure=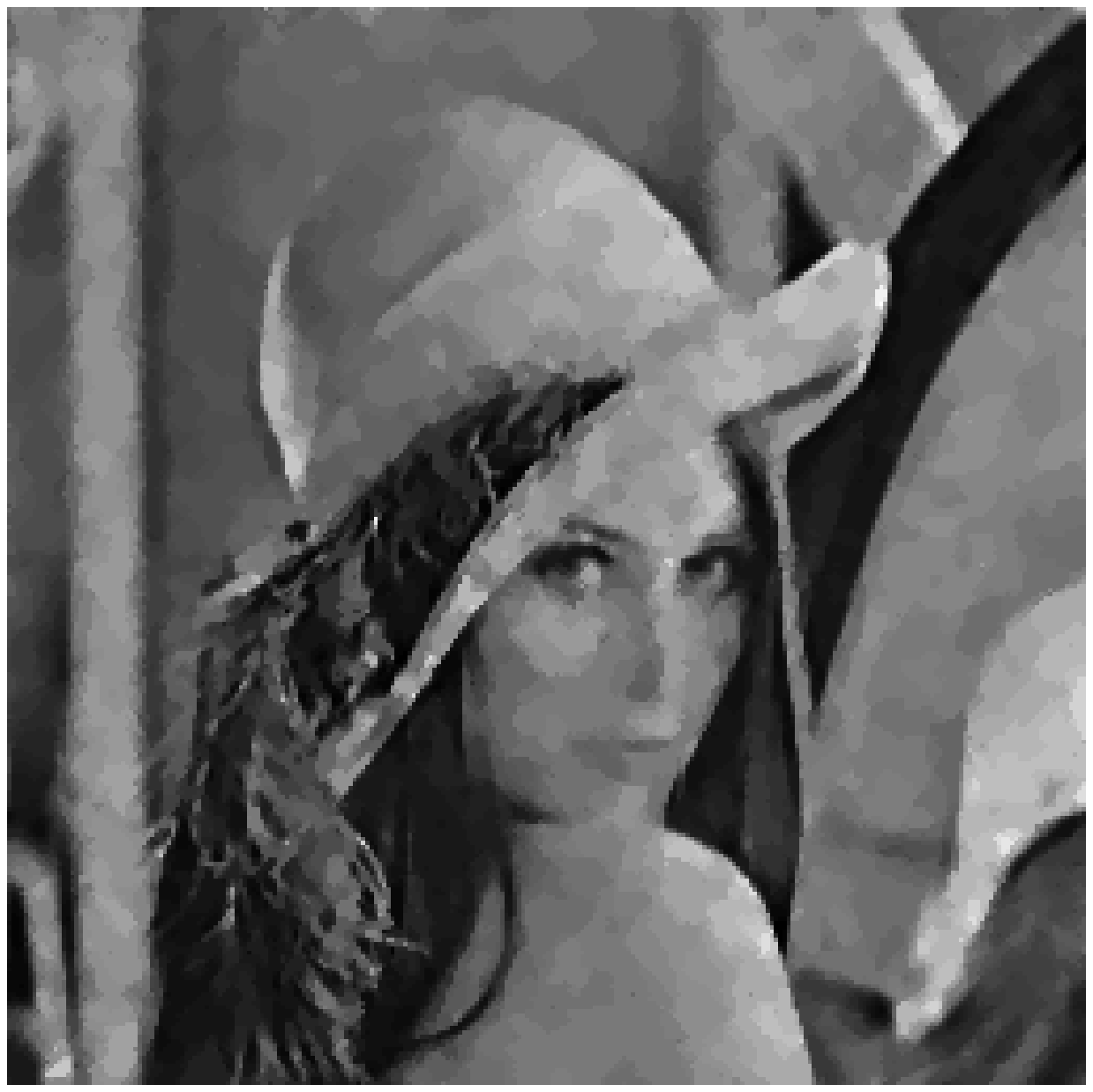,width=5cm,height=5cm}\\
{(d) Stein-block thresholding \cite{Chesneau08}}&{(e) AA algorithm \cite{AubertAujol08}}&{(f) Our method}\\
{{\sc psnr}=25.49 d\sc{b,~mae}=9.45 }&{{\sc psnr}=25.37 d\sc{b,~mae}=9.41 }&{{\sc psnr}=26.05 d\sc{b,~mae}=8.8 }\\
\epsfig{figure=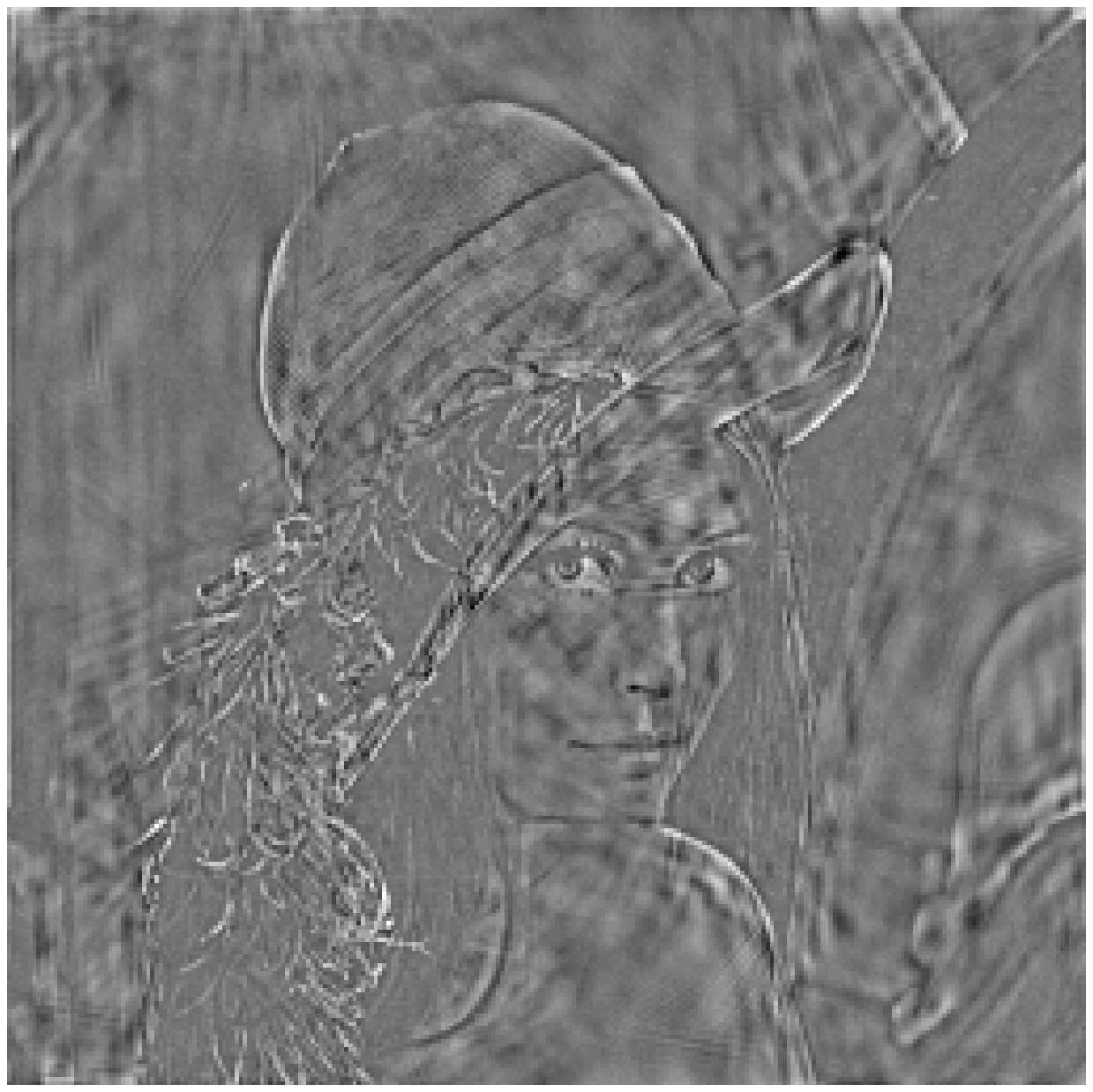,width=5cm,height=5cm}&
\epsfig{figure=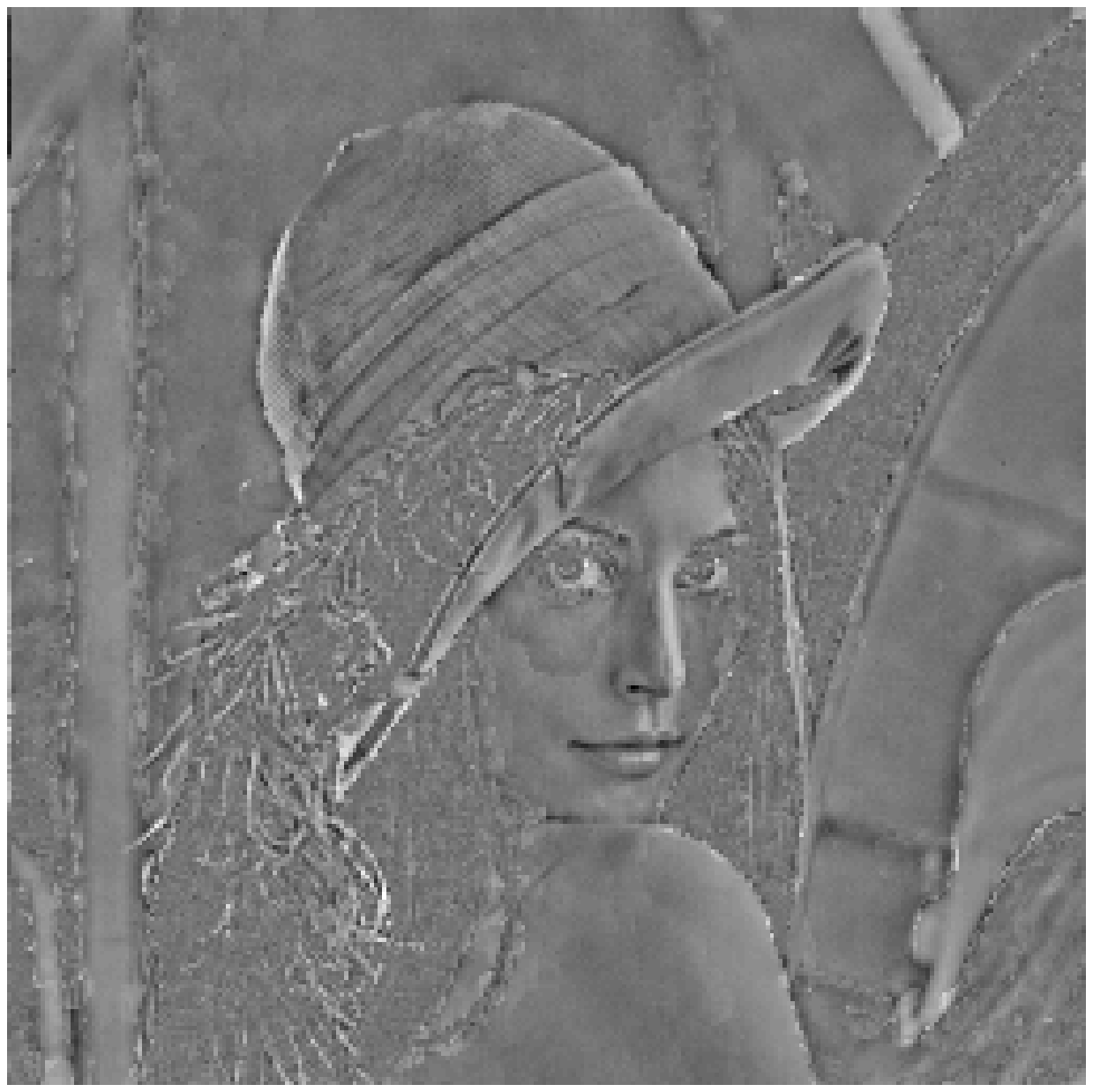,width=5cm,height=5cm}&
\epsfig{figure=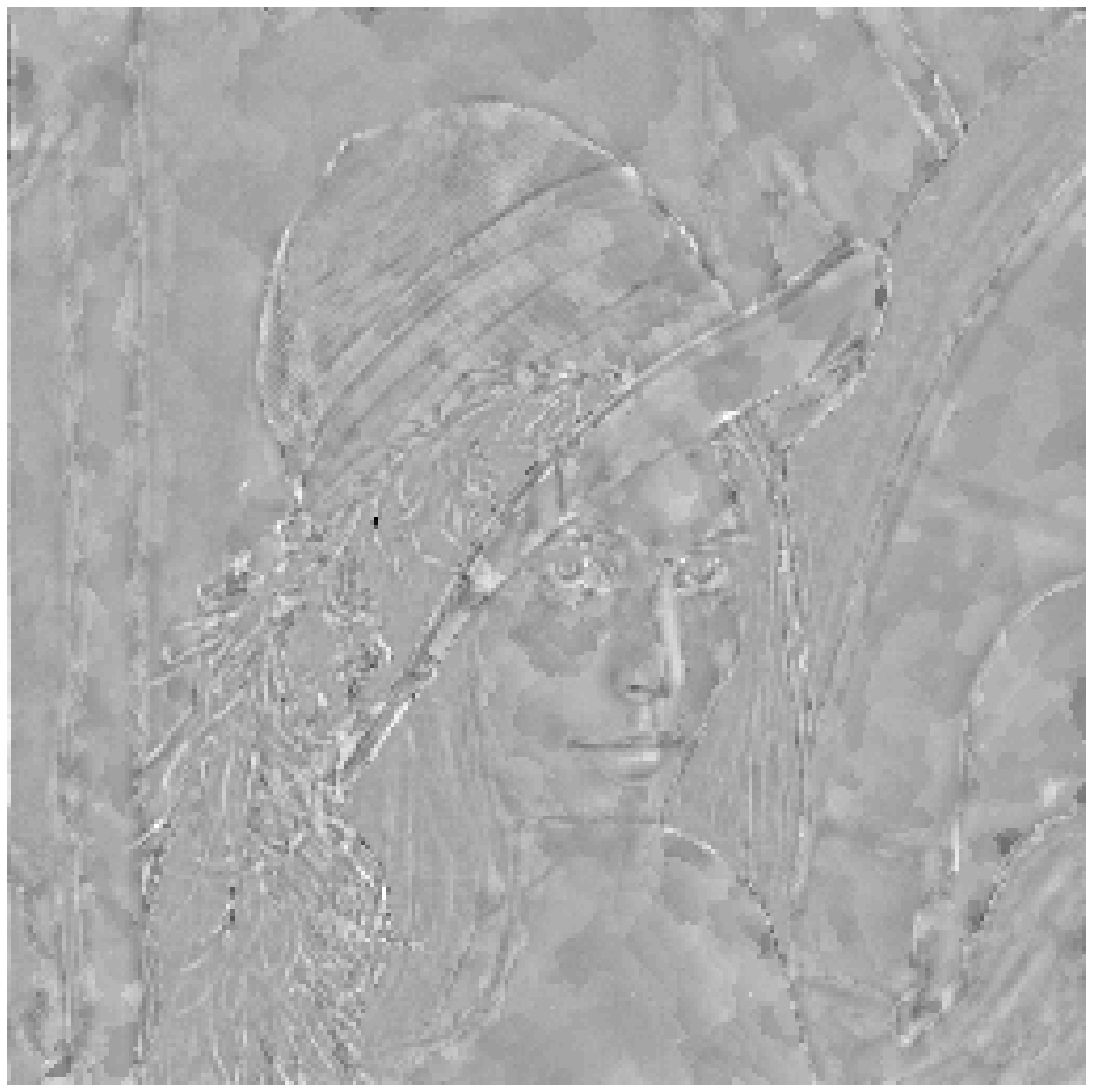,width=5cm,height=5cm}
\end{tabular}
\caption{Comparative restoration of the noisy Lena in (b) using modern methods.
Note that the algorithm in (c) is initialized with the log-data and that
the restoration in (d) is done in the curvelet domain.
The images on the last row show the error (restored $-$ original) for (d), (e) and (f).}
\label{fig:lena}
\end{figure}

\noindent\begin{figure}[ht]\vspace{-0.8cm}
\begin{tabular}{ccc}
\epsfig{figure=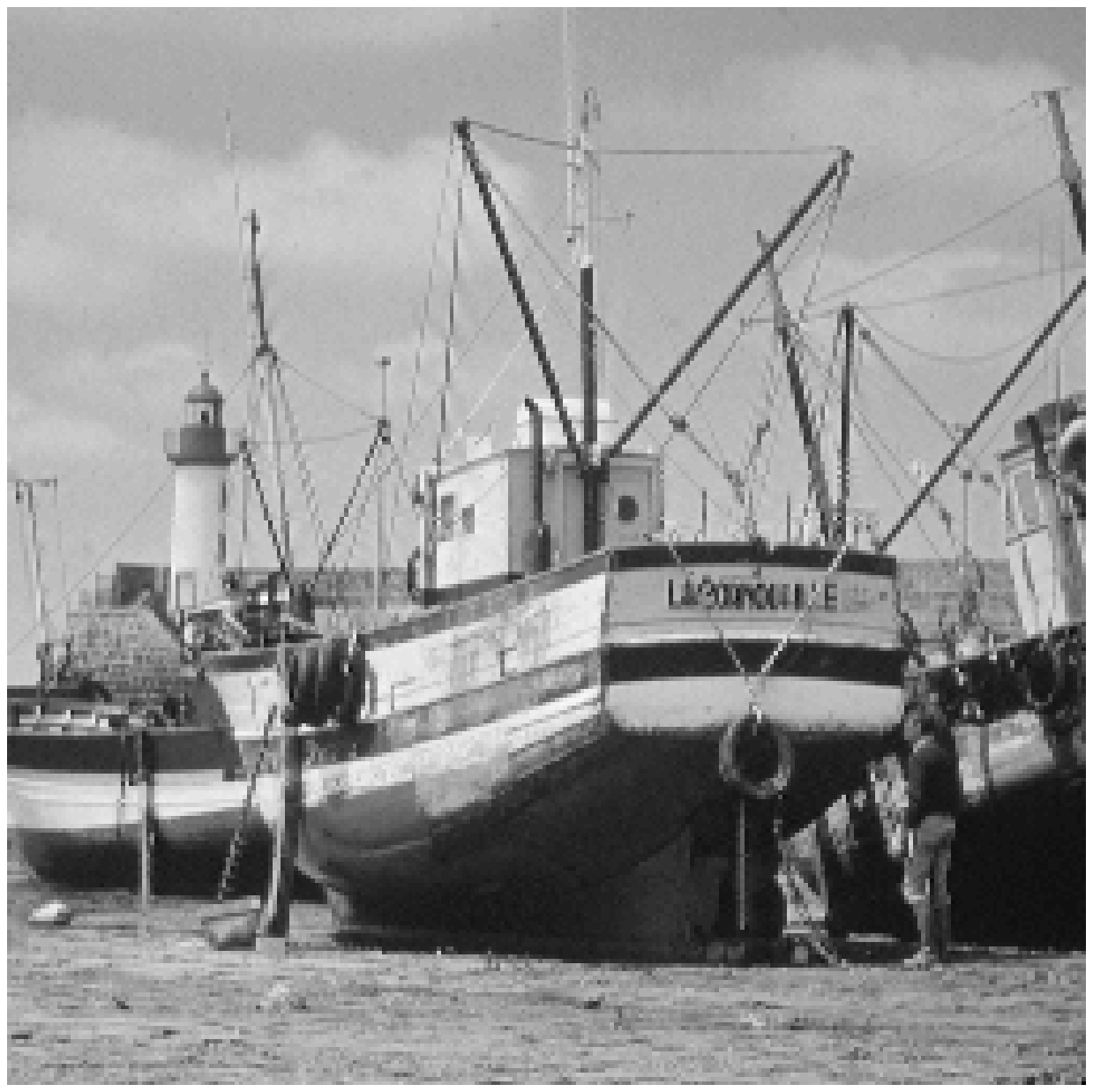,width=5.1cm,height=5.1cm}&
\epsfig{figure=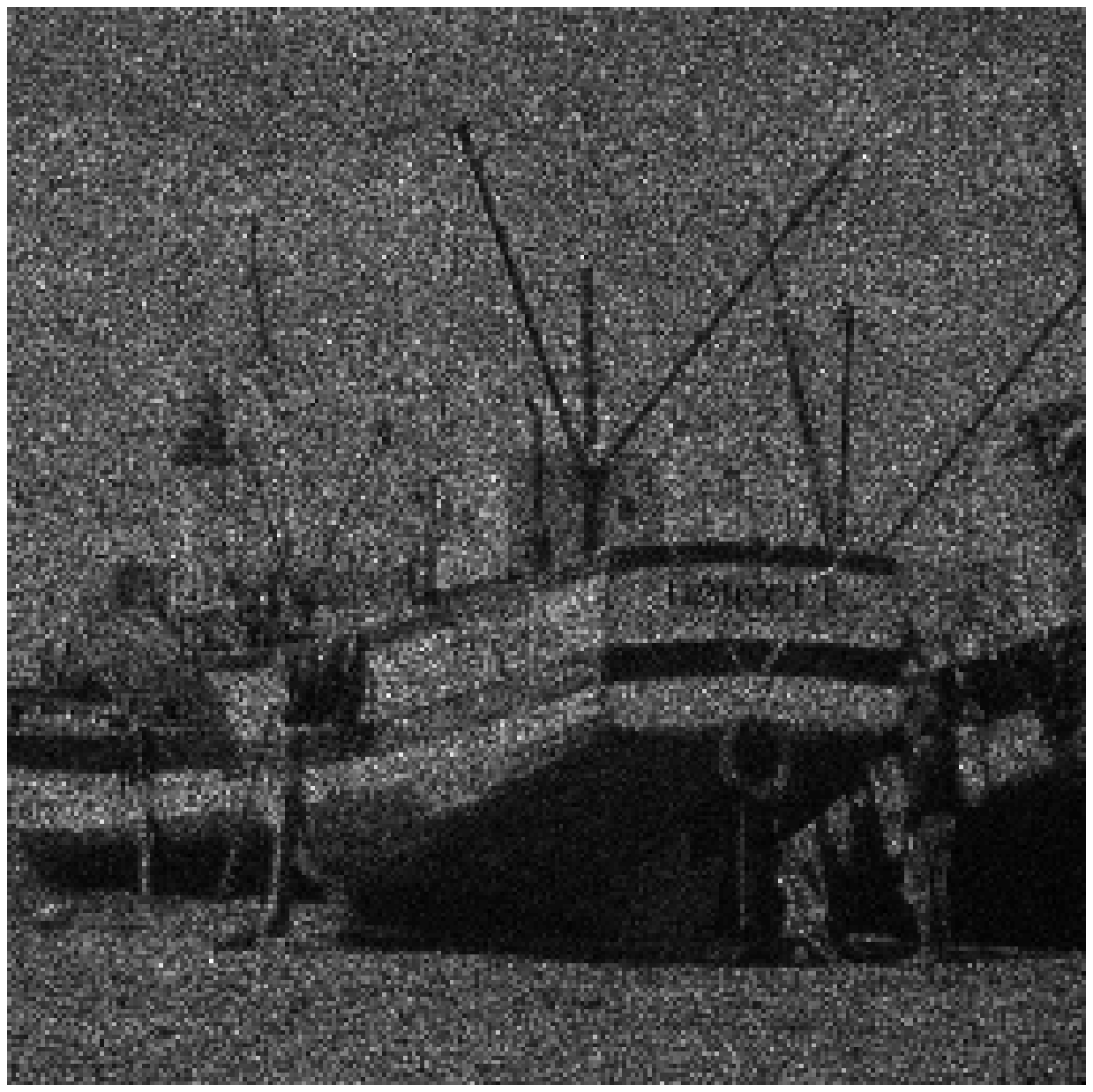,width=5.1cm,height=5.1cm}&
\epsfig{figure=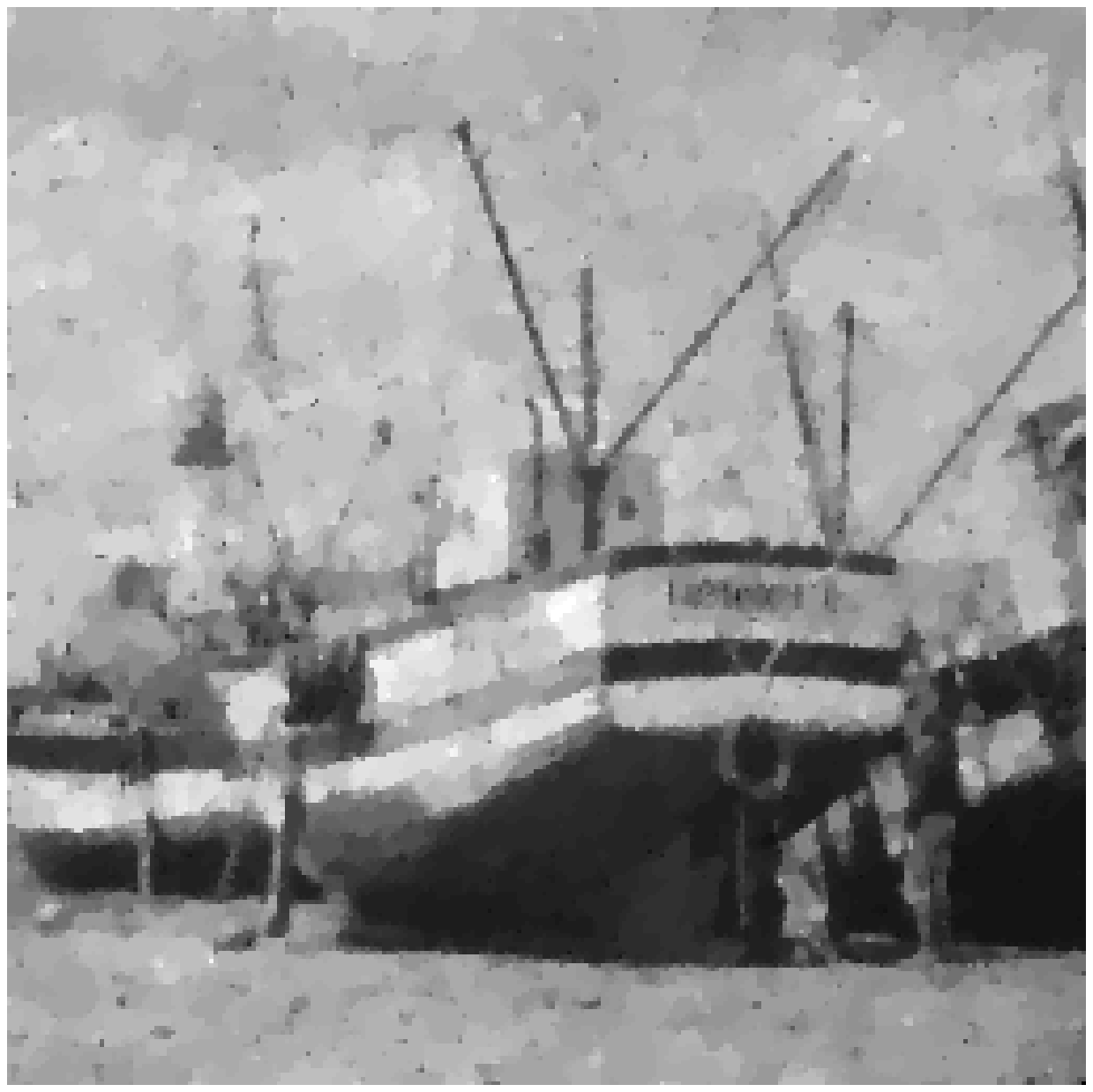,width=5.1cm,height=5.1cm}\\
{(a) Boat ($256\times 256$)---original}&{(b) Noisy: $\mu=1, ~K=10$}&{(c) L2-TV}\\
&{see \eqref{expo}-\eqref{Gamma}}&{{\sc psnr}=24.118d\sc{b,~mae}=10.202  }\\
\epsfig{figure=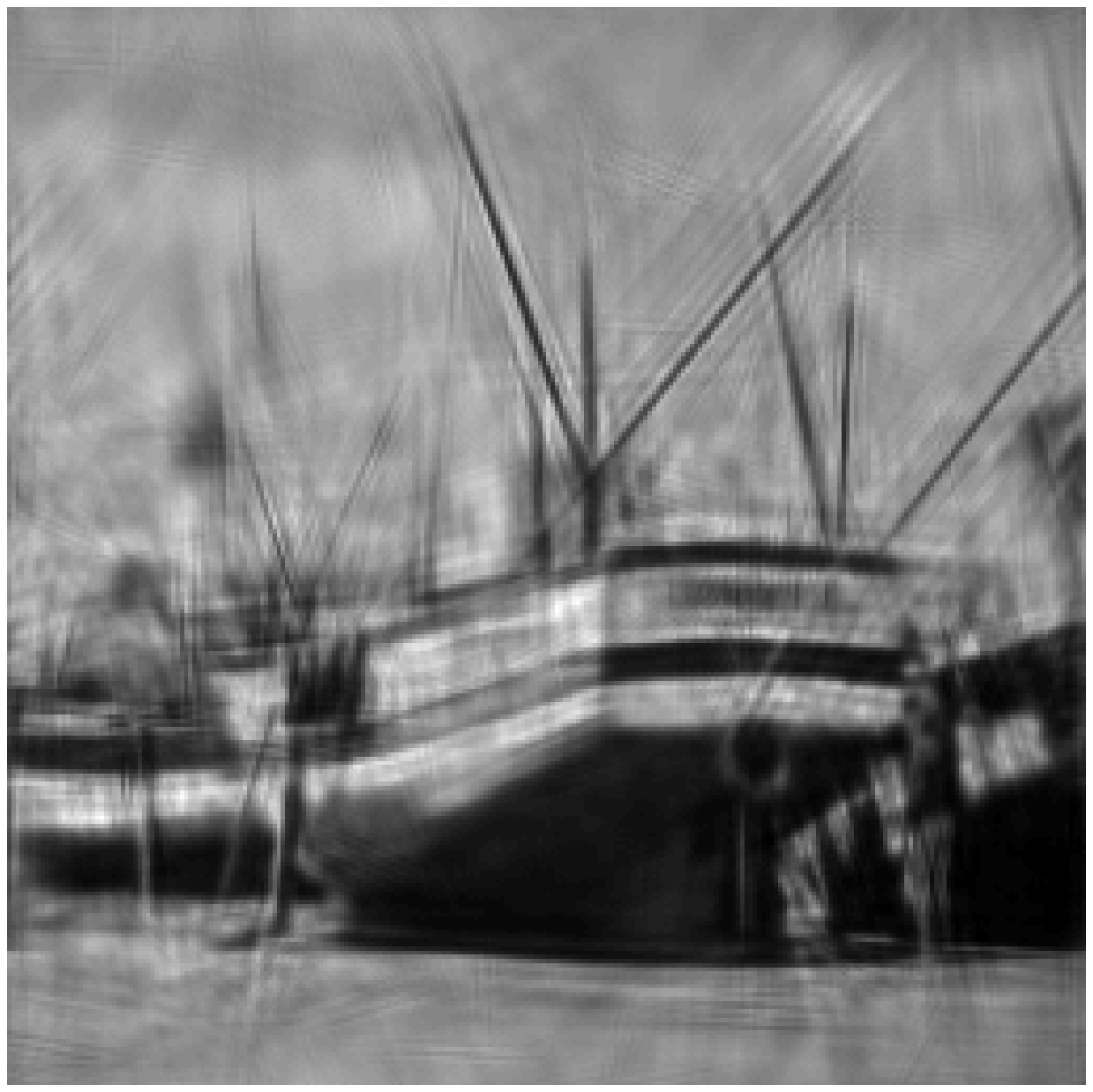,width=5.1cm,height=5.1cm}&
\epsfig{figure=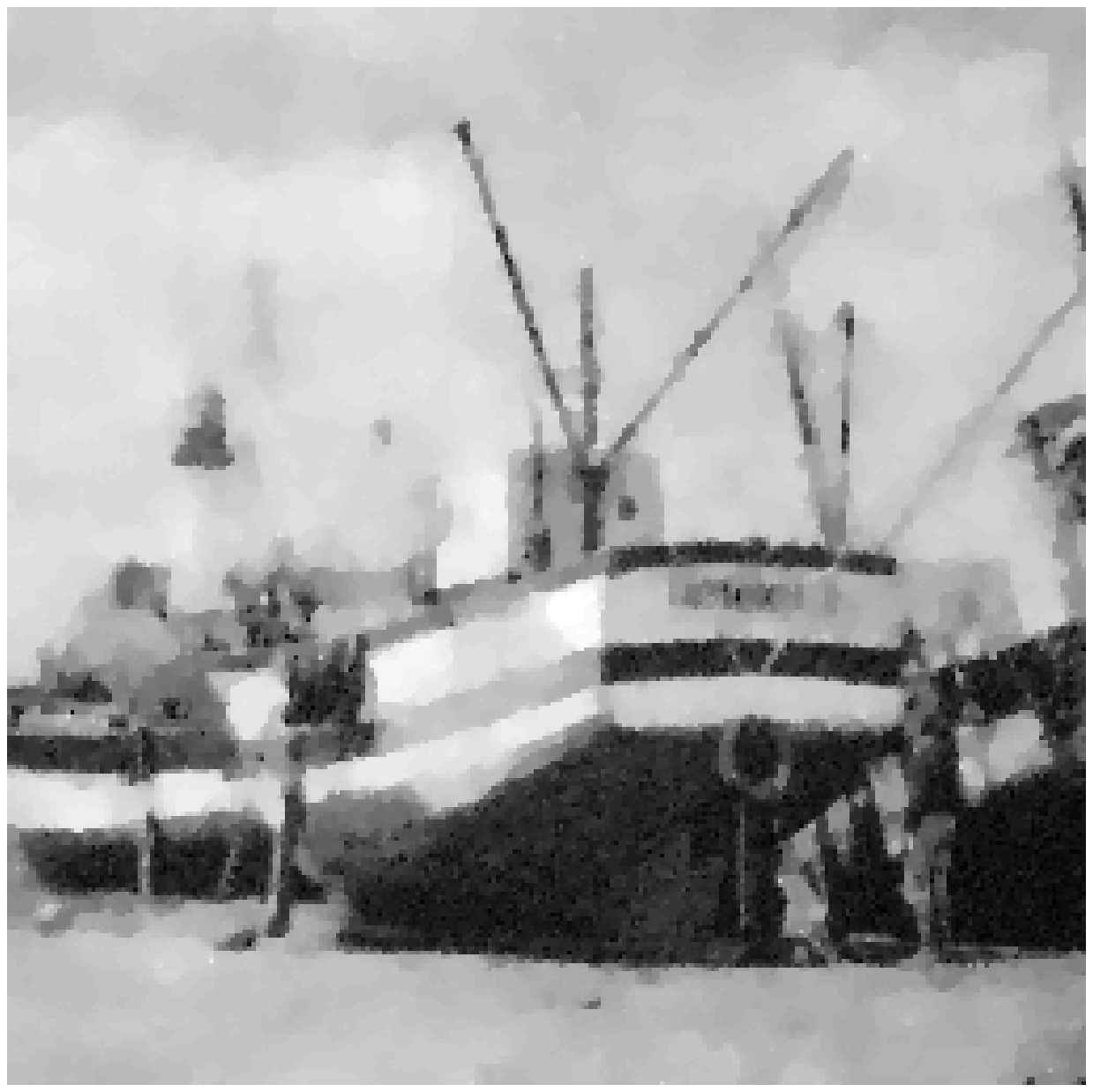,width=5.1cm,height=5.1cm}&
\epsfig{figure=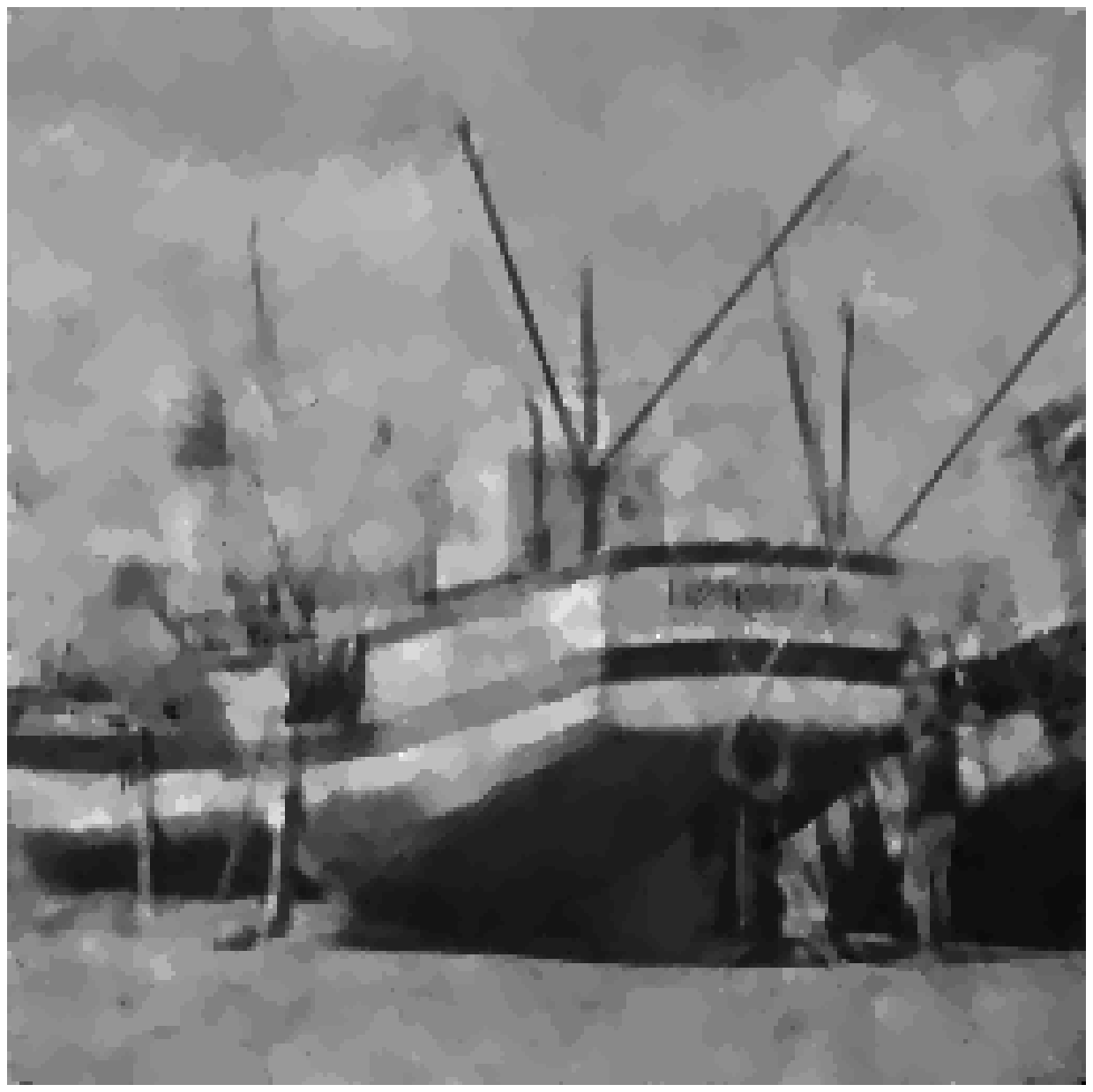,width=5.1cm,height=5.1cm}\\
{(d) Stein-block thresholding \cite{Chesneau08}}&{(e) AA algorithm \cite{AubertAujol08}}&{(f) Our method}\\
{{\sc psnr}=23.57d\sc{b,~mae}=10.98}&{{\sc psnr}=23.36d\sc{b,~mae}=11.08}&{{\sc psnr}=24.12d\sc{b,~mae}=10.2}\\
\epsfig{figure=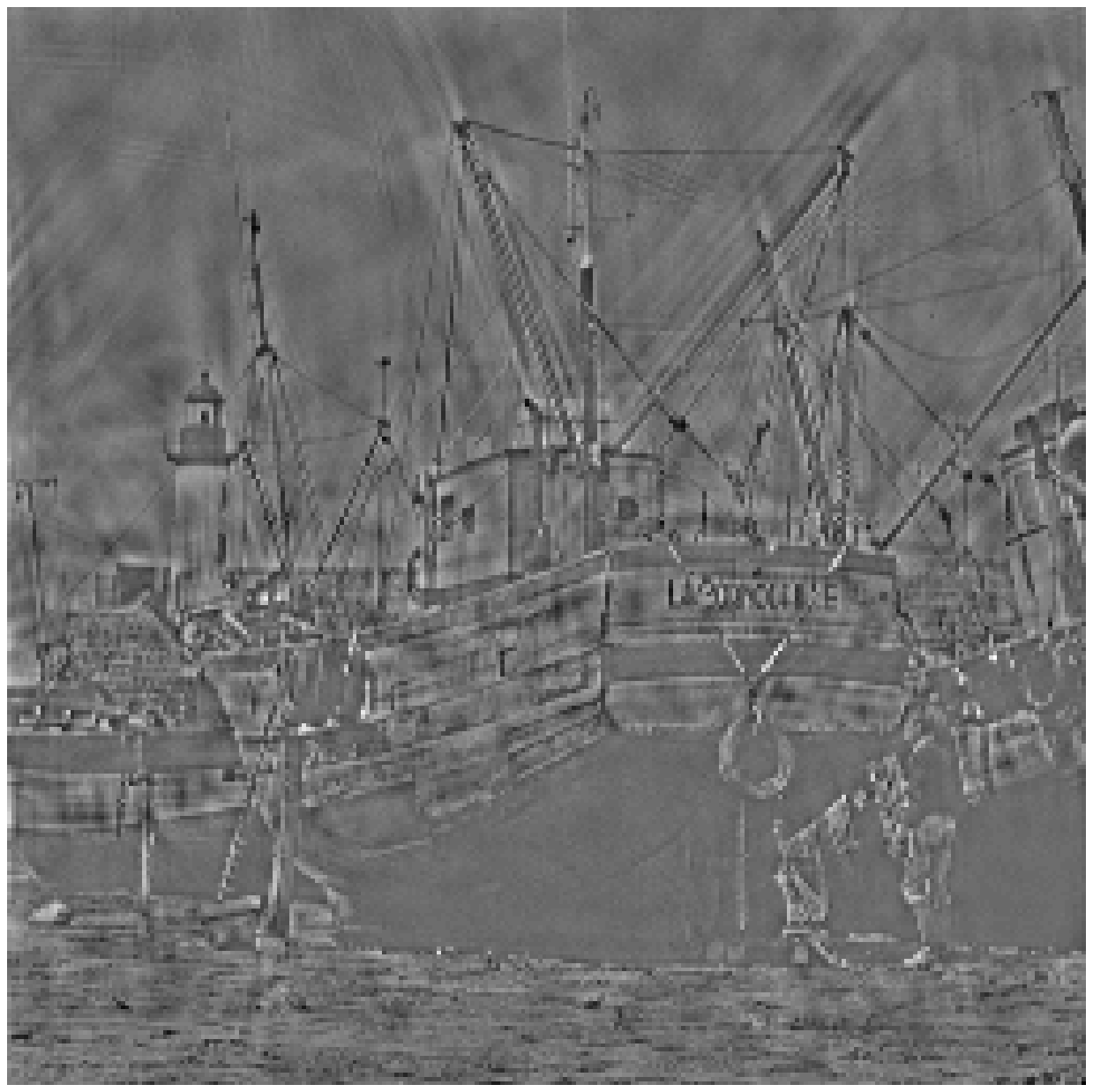,width=5.1cm,height=5.1cm}&
\epsfig{figure=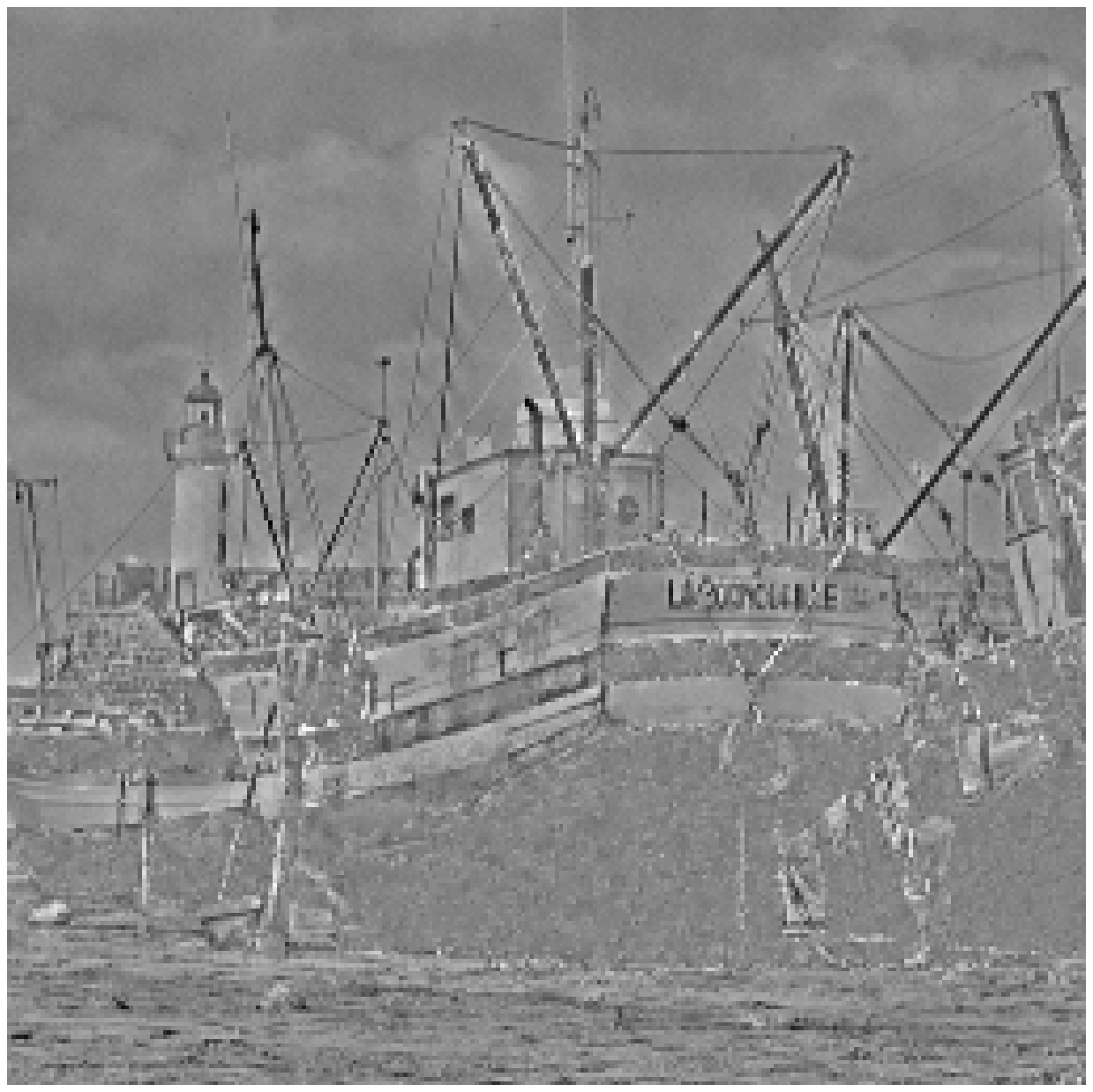,width=5.1cm,height=5.1cm}&
\epsfig{figure=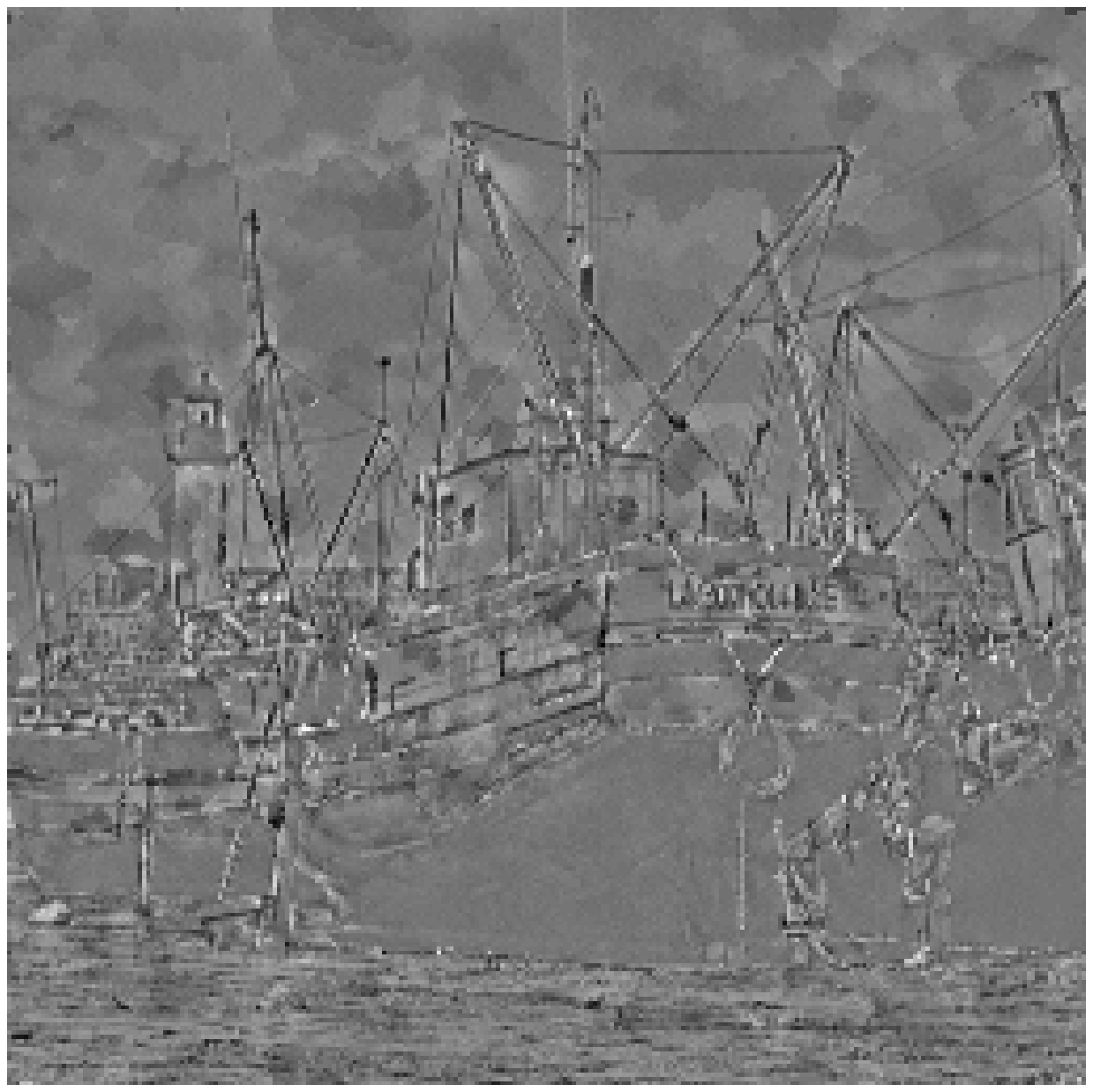,width=5.1cm,height=5.1cm}
\end{tabular}
\caption{Restoration of (b) using contemporary methods.
The last row shows the error images, namely (restored $-$ original) for (d), (e) and (f).}
\label{fig:boat}
\end{figure}
\section{Conclusions}\label{conclusion}

This work proposes quite an original, efficient and fast method for multiplicative noise removal.
The latter is a difficult problem that arises in various applications relevant to active imaging system, such as
laser imaging, ultrasound imaging, SAR and many others. Multiplicative noise contamination involves
inherent difficulties that severely restrict the main restoration algorithms.

The main ingredients of our method are: (1) consider the log-data to restore a log-image; (2) preprocess the log-fata
using and under-optimal hard-thresholding of its tight frame coefficients; (3) restore the log-image using a hybrid
criterion composed of an $\ell^1$ data-fitting for the coefficients and a TV regularization in the log-image domain;
(4) restore the sought-after image using an exponential transform along with a pertinent bias correction.
The resultant algorithm is fast, its consistency and convergence are proved theoretically.

The obtained numerical results are really encouraging since they outperform the most recent methods in this field.

\bibliographystyle{plain}
\markboth{}{}

\small

\input{bibdef.tex}

\input{MultNoise.bbl}

\end{document}

%% file: bibdef.tex
%
\def\jan{Jan. }
\def\feb{Feb. }
\def\mar{Mar. }
\def\apr{Apr. }
\def\may{May }
\def\jun{June }
\def\jul{July }
\def\aug{Aug. }
\def\sep{Sep. }
\def\oct{Oct. }
\def\nov{Nov. }
\def\dec{Dec. }
\def\Jan{Jan. }
\def\Feb{Feb. }
\def\Mar{Mar. }
\def\Apr{Apr. }
\def\May{May }
\def\Jun{June }
\def\Jul{July }
\def\Aug{Aug. }
\def\Sep{Sep. }
\def\Oct{Oct. }
\def\Nov{Nov. }
\def\Dec{Dec. }
\def\sub{submitted to }


\def\AsAs{Astrononmy and Astrophysics}
\def\AAP{Advances in Applied Probability}
\def\ABE{Annals of Biomedical Engineering}
\def\ABE{Annals of Biomedical Engineering}
\def\ACHA{Applied and Computational Harmonic Analysis}
\def\AISM{Annals of Institute of Statistical Mathematics}
\def\AMS{Annals of Mathematical Statistics}
\def\AO{Applied Optics}
\def\AP{The Annals of Probability}
\def\ARAA{Annual Review of Astronomy and Astrophysics}
\def\AST{The Annals of Statistics}
\def\AT{Annales des T\'el\'ecommunications}
\def\BMC{Biometrics}
\def\BMK{Biometrika}
\def\CPAM{Communications on Pure and Applied Mathematics}
\def\EMK{Econometrica}
\def\CRAS{Compte-rendus de l'acad\'emie des sciences}
\def\CVGIP{Computer Vision and Graphics and Image Processing}
\def\GJRAS{Geophysical Journal of the Royal Astrononomical Society}
\def\GSC{Geoscience}
\def\GPH{Geophysics}
\def\GRETSI#1{Actes du #1$^{\mbox{e}}$ Colloque GRETSI}
\def\CGIP{Computer Graphics and Image Processing}
\def\ICASSP{Proceedings of the IEEE Int. Conf. on Acoustics,
    Speech and Signal Processing}
\def\ICEMBS{Proceedings of IEEE EMBS}
\def\ICIP{Proceedings of the IEEE International Conference on Image Processing}
\def\EUSIPCO{Proceedings of European Signal Processing Conference}
\def\SSAP{Proceedings of the IEEE Statistical Signal and Array Processing}
\def\ieeP{Proceedings of the IEE}
\def\ieeeAC{IEEE Transactions on Automatic and Control}
\def\ieeeAES{IEEE Transactions on Aerospace and Electronic Systems}
\def\ieeeAP{IEEE Transactions on Antennas and Propagation}
\def\ieeeAPm{IEEE Antennas and Propagation Magazine}
\def\ieeeASSP{IEEE Transactions on Acoustics Speech and Signal Processing}
\def\ieeeBME{IEEE Transactions on Biomedical Engineering}
\def\ieeeCS{IEEE Transactions on Circuits and Systems}
\def\ieeeCT{IEEE Transactions on Circuit Theory}
\def\ieeeC{IEEE Transactions on Communications}
\def\ieeeGE{IEEE Transactions on Geoscience and Remote Sensing}
\def\ieeeGEE{IEEE Transactions on Geosciences Electronics}
\def\ieeeIP{IEEE Transactions on Image Processing}
\def\ieeeIT{IEEE Transactions on Information Theory}
\def\ieeeMI{IEEE Transactions on Medical Imaging}
\def\ieeeMTT{IEEE Transactions on Microwave Theory and Technology}
\def\ieeeM{IEEE Transactions on Magnetics}
\def\ieeeNS{IEEE Transactions on Nuclear Sciences}
\def\ieeePAMI{IEEE Transactions on Pattern Analysis and Machine Intelligence}
\def\ieeeP{Proceedings of the IEEE}
\def\ieeeRS{IEEE Transactions on Radio Science}
\def\ieeeSMC{IEEE Transactions on Systems, Man and Cybernetics}
\def\ieeeSP{IEEE Transactions on Signal Processing}
\def\ieeeSSC{IEEE Transactions on Systems Science and Cybernetics}
\def\ieeeSU{IEEE Transactions on Sonics and Ultrasonics}
\def\ieeeUFFC{IEEE Transactions on Ultrasonics Ferroelectrics and Frequency Control}
\def\IJC{International Journal of Control}
\def\IJCV{International Journal of Computer Vision}
\def\IJIST{International Journal of Imaging Systems and Technology}
\def\IP{Inverse Problems}
\def\ISR{International Statistical Review}
\def\IUSS{Proceedings of International Ultrasonics Symposium}
\def\JAPH{Journal of Applied Physics}
\def\JAP{Journal of Applied Probability}
\def\JAS{Journal of Applied Statistics}
\def\JASA{Journal of Acoustical Society America}
\def\JASAS{Journal of American Statistical Association}
\def\JBME{Journal of Biomedical Engineering}
\def\JCAM{Journal of Computational and Applied Mathematics}
\def\JCP{Journal of Computational Physics}
\def\JEWA{Journal of Electromagnetic Waves and Applications}
\def\JMIV{Journal of Mathematical Imaging and Vision}
\def\JMO{Journal of Modern Optics}
\def\JNDE{Journal of Nondestructive Evaluation}
\def\JMP{Journal of Mathematical Physics}
\def\JOSA{Journal of the Optical Society of America}
\def\JP{Journal de Physique}
\def\JRSSA{Journal of the Royal Statistical Society A}
\def\JRSSB{Journal of the Royal Statistical Society B}
\def\JRSSC{Journal of the Royal Statistical Society C}
\def\JSPI{Journal of Statistical Planning and Inference}
\def\JTSA{Journal of Time Series Analysis}
\def\JVCIR{Journal of Visual Communication and Image Representation}
\def\KAP{Kluwer \uppercase{A}cademic \uppercase{P}ublishers}
\def\MNAS{Mathematical Methods in Applied Science}
\def\MNRAS{Monthly Notices of the Royal Astronomical Society}
\def\MP{Mathematical Programming}
    \def\NSIP{NSIP}  
\def\OC{Optics Communication}
\def\PRA{Physical Review A}
\def\PRB{Physical Review B}
\def\PRC{Physical Review C}
\def\PRD{Physical Review D}
\def\PRL{Physical Review Letters}
\def\RGSP{Review of Geophysics and Space Physics}
\def\RPA{Revue de Physique Appliqu\'e}
\def\RS{Radio Science}
\def\SP{Signal Processing}
\def\siamAM{SIAM Journal on Applied Mathematics}
\def\siamCO{SIAM Journal on Control and Optimization}
\def\siamC{SIAM Journal on Control}
\def\siamJO{SIAM Journal on Optimization}
\def\siamMA{SIAM Journal on Mathematical Analysis}
\def\siamNA{SIAM Journal on Numerical Analysis}
\def\siamSC{SIAM Journal on Scientific Computing}
\def\siamMMS{SIAM Journal on Multiscale Modeling and Simulation}
\def\siamR{SIAM Review}
\def\SSR{Stochastics and Stochastics Reports}
\def\TPA{Theory of Probability and its Applications}
\def\TMK{Technometrics}
\def\TS{Traitement du Signal}
\def\UCMMP{U.S.S.R. Computational Mathematics and Mathematical Physics}
\def\UMB{Ultrasound in Medecine and Biology}
\def\US{Ultrasonics}
\def\USI{Ultrasonic Imaging}